\documentclass[letterpaper, 11pt,reqno]{amsart}
\usepackage{mathtools}
\usepackage{amsmath}
\usepackage{amssymb}
\usepackage{yhmath}
\usepackage{graphicx,overpic}
\usepackage{mathrsfs}
\usepackage{bbm}
\usepackage[dvipsnames]{xcolor}
\usepackage{tikz-cd}
\usepackage{tikz}
\usetikzlibrary{patterns}
\usepackage[
    hypertexnames=false,
    colorlinks=true,    
    linkcolor=NavyBlue,
    hypertexnames=false,
    filecolor=black,      
    urlcolor=blue,
    citecolor=OrangeRed]{hyperref}

\usepackage{verbatim}
\usepackage{float}

\DeclareMathAlphabet{\mathpzc}{OT1}{pzc}{m}{it}

\usepackage{thmtools}
\usepackage{thm-restate}

\usepackage{caption}

\usepackage[]{cleveref}
\newtheorem{theorem}{Theorem}[section]

\newtheorem*{claim*}{Claim}
\newtheorem{claim}{Claim}

\newtheorem{lemma}[theorem]{Lemma}
\newtheorem{lem}[theorem]{Lemma}
\newtheorem{corollary}[theorem]{Corollary}

\newtheorem{cor}[theorem]{Corollary}

\newtheorem{proposition}[theorem]{Proposition}

\newtheorem{prop}[theorem]{Proposition}

\theoremstyle{definition}

\newtheorem{Def}[theorem]{Definition}

\theoremstyle{remark}

\newtheorem{Rmk}[theorem]{Remark}

\numberwithin{equation}{section}

\def\geo{\partial_\infty}

\newcommand{\op}{\operatorname}
\newcommand{\Ad}{\op{Ad}}

\newcommand{\be}{\begin{equation}}
\newcommand{\ee}{\end{equation}}
\newcommand{\Ga}{\Gamma}

\newcommand{\R}{\mathbb R}
\renewcommand{\H}{\mathbb H}
\newcommand{\Z}{\mathbb Z}
\newcommand{\N}{\mathbb N}
\newcommand{\ga}{\gamma}

\newcommand{\La}{\Lambda}
\newcommand{\inte}{\op{int}}
\newcommand{\ba}{\backslash}

\newcommand{\cal}{\mathcal}
\newcommand{\br}{\mathbb R}
\newcommand{\SO}{\op{SO}}

\newcommand{\PSL}{\op{PSL}}

\newcommand{\core}{\op{core}}

\newcommand{\bH}{\mathbb H}

\newcommand{\G}{\Gamma}
\newcommand{\m}{\mathsf{m}}

\newcommand{\T}{\mathscr T}
\renewcommand{\frak}{\mathfrak}

\newcommand{\e}{\varepsilon}

\newcommand{\fa}{\mathfrak a}

\newcommand{\Lie}{\op{Lie}}

\newcommand{\supp}{\op{supp}}

\newcommand{\diag}{\op{diag}}

\renewcommand{\epsilon}{\e}

\newcommand{\SL}{\op{SL}}

\def\g{\gamma}

\newcommand{\fh}{\mathfrak{h}}

\newcommand{\Hom}{\op{Hom}}

\renewcommand{\b}{\mathsf b}
\renewcommand{\T}{\op{T}}
\newcommand{\mc}{\mathsf c}
\newcommand{\ms}{\mathsf s}
\DeclareMathOperator{\wdt}{wd}
\DeclareMathOperator{\ldim}{\underline{dim}}

\title[Fractal closures of geodesic planes]{Fractal closures of geodesic planes \\ in Hitchin manifolds}

\author{Subhadip Dey}
\address{School of Mathematics, Tata Institute of Fundamental Research, Mumbai
}
\email{subhadip@math.tifr.res.in}

\author{Hee Oh}
\address{Department of Mathematics, Yale University, New Haven, CT}
\email{hee.oh@yale.edu}
\thanks{
 Oh is partially supported by the NSF grant No. DMS-2450703.}

\begin{document}
\begin{abstract} 
Ratner’s theorem implies topological rigidity of immersed totally geodesic subspaces of noncompact type in finite-volume locally symmetric spaces. In higher rank and infinite volume, however, counter-examples to this rigidity have remained elusive.

We construct the first such examples using \emph{floating geodesic planes}. Specifically, we exhibit a Zariski-dense Hitchin surface group
$\Gamma < \mathrm{SL}_3(\mathbb{R})$ such that the Hitchin manifold
$\Gamma \backslash \mathrm{SL}_3(\mathbb{R}) / \mathrm{SO}(3)$
contains immersed floating geodesic planes whose closures are fractal, with non-integer Hausdorff dimensions accumulating at $2$. Moreover, $\Gamma$ can be chosen inside $\mathrm{SL}_3(\mathbb{Z})$.

\end{abstract}

\maketitle

\tableofcontents

\section{Introduction}
The study of orbit closures for actions of subgroups generated by unipotent elements has been one of the central themes in homogeneous dynamics.
A landmark result is Ratner’s 1991 theorem \cite{Ra}, which resolved the conjecture of Raghunathan. It says the following: if $G$ is
a connected semisimple Lie group $G$ and $\Gamma < G$ is a lattice (a discrete subgroup of finite covolume), then for any connected subgroup $H<G$ generated by unipotent elements,
the closure of every $H$-orbit in $\Ga\ba G$ is itself a homogeneous subspace, namely, a subspace of the form $xL$
where $L < G$ is a Lie subgroup containing $H$ and $x$ is a point in $\Ga\ba G$. 
This result implies the following topological rigidity of geodesic planes: in any locally symmetric space of noncompact type and finite volume, the closure of an immersed totally geodesic subspace of noncompact type\footnote{this means the image of a totally geodesic immersion of a locally symmetric space of noncompact type} is an immersed submanifold. 

An important special case was obtained earlier by  Margulis and Dani-Margulis (1989): in $\SL_3(\mathbb{Z}) \backslash \SL_3(\mathbb{R})$, any orbit of $\SO(2,1)$ is either closed or dense (\cite{Ma}, \cite{DM}). This implies that in the associated locally symmetric space
$\SL_3(\Z)\ba \SL_3(\br)/\SO(3)$,
 every {\em irreducible} totally geodesic plane is either properly immersed or dense. The closed or dense dichotomy has a far-reaching consequence. In fact, Margulis's proof of the Oppenheim conjecture that for any irrational indefinite quadratic form $Q(x_1, \dots, x_n)$ with $n \ge 3$, the set of values $Q(\mathbb{Z}^n)$ is dense in $\mathbb{R}$ was his famous application of this result \cite{Ma}.

\medskip 
In the infinite-volume setting, the geometry of the ambient space plays a decisive role. For
convex cocompact acylindrical hyperbolic $3$-manifolds,
 McMullen–Mohammadi–Oh  proved that geodesic planes inside the interior of the convex core is either properly immersed or dense (\cite{MMO}, \cite{MMO2}), and Benoist–Oh  extended this to geometrically finite acylindrical manifolds \cite{BO}. In higher dimensions, Lee–Oh gave a complete classification of geodesic plane closures for convex cocompact real hyperbolic manifolds with Fuchsian ends (\cite{LeO}; see also the survey paper \cite{Oh}).

The picture changes dramatically once we leave  the acylindrical setting. Using bending deformations, McMullen–Mohammadi–Oh \cite{MMO} constructed quasi-Fuchsian hyperbolic $3$-manifolds that contain chaotic geodesic planes, arising from planes orthogonal to a chaotic geodesic of a closed hyperbolic surface. Here ``chaotic'' means that the closures of these planes have non-integer Hausdorff dimension. This stands in sharp contrast with the acylindrical case: acylindrical hyperbolic manifolds exhibit strong geometric constraints that enforce a certain $k$-thickness property for every circular slice of the limit set, whereas quasi-Fuchsian manifolds may support much thinner circular slices.  Geometrically, thin circular slices of the limit set translate into
scant recurrence for unipotent flows, obstructing the standard
homogeneous dynamics approach to orbit closures.

 In higher rank and infinite volume, examples demonstrating the failure of topological rigidity of orbit closures have remained elusive. In this paper, we provide the first such examples.

 \begin{figure} \begin{center}
  \includegraphics [height=6cm]{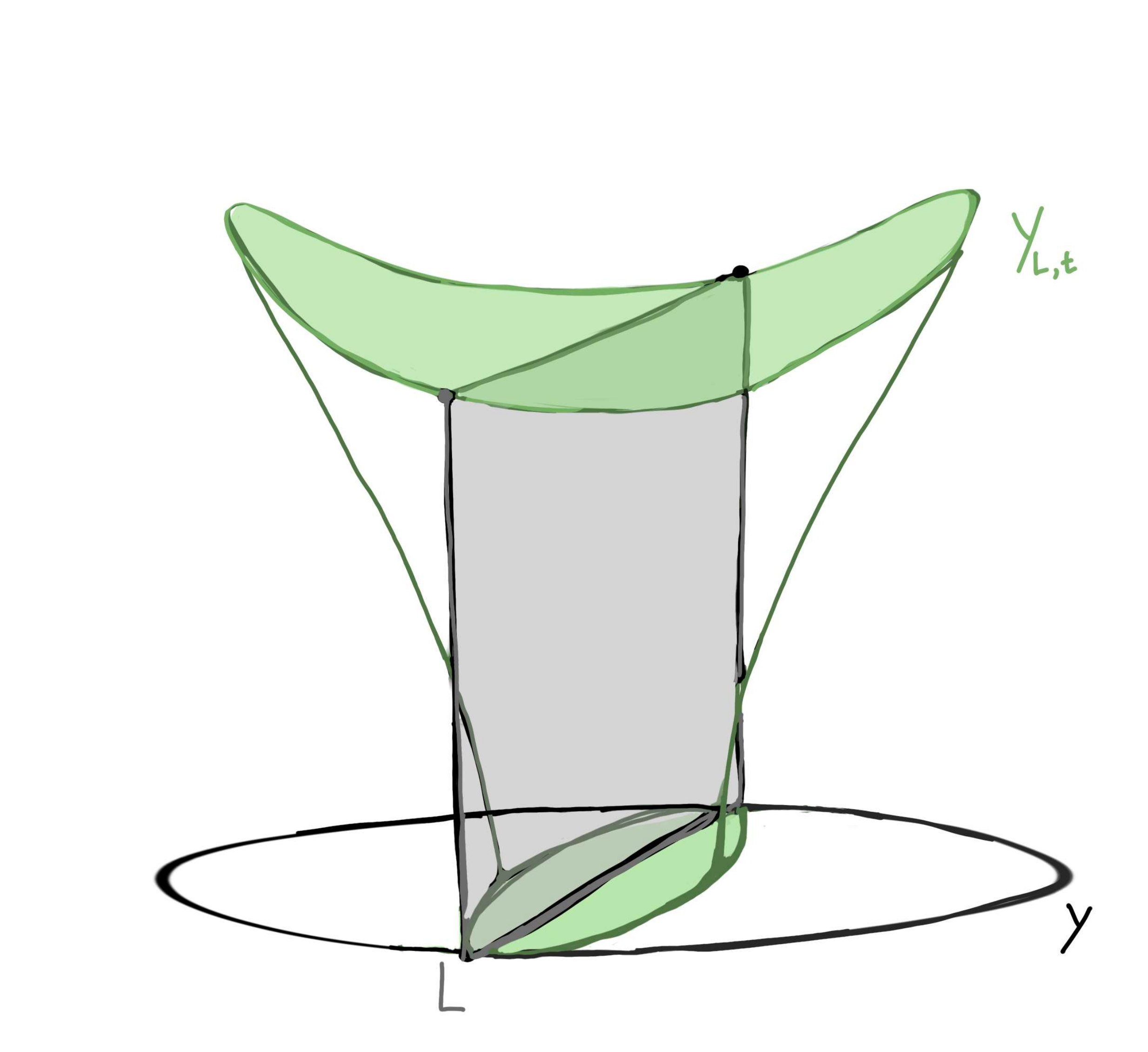}  \end{center}\caption{Floating geodesic plane}
  \label{fig:floatingplane}
  \end{figure} 
 
 \bigskip 
Let
$$G=\SL_3(\br), \; \; K=\SO(3), \; \; X=G/K,   $$
so that $X$ is the Riemmanian symmetric space of unimodular positive-definite symmetric matrices.
We consider 
$$H=\SO(2,1)^\circ ,\; \;\; Y= H\cdot o \subset X, \quad{\text{where } o=[K]\in X} ;$$
$Y$  is an irreducible totally geodesic plane passing through the basepoint $o$.
Our primary objects are what we call 
$$\text{{\it floating geodesic planes}}$$
defined as follows:
Given a complete geodesic $L\subset Y$, there exists a unique maximal flat $\cal F\subset X$ that intersects $Y$ orthogonally along $L$. The plane $Y$ can then be shifted away from itself along $\cal F$, producing ultra-parallel copies that ``float'' in the ambient space. More precisely, for any $t>0$, the floating geodesic plane $Y_{L, t}$ is defined as the translate of $Y$
along the orthogonal direction in $\cal F$ at distance $t$.  Equivalently, if $\xi_t=a_t o$ is the unit speed geodesic in $\cal F$ orthogonal to $L$ at $o$ for a one-parameter diagonalizable subgroup $\{a_t:t\in \br\}$, then $$Y_{L, t}=a_t Y;
$$ see \Cref{fig:floatingplane}.  We also refer to the image of $Y_{L,t}$ in any quotient manifold $\Gamma\ba X$ as a floating geodesic plane.

In hyperbolic spaces, there are no flats of dimension larger than one, so a geodesic plane cannot be displaced in this way. The existence of floating geodesic planes is therefore a phenomenon that only appears in higher rank.

The main result of this paper is as follows:

\begin{theorem}\label{m1}\label{max}
    There exists a Zariski dense Hitchin surface subgroup $\Gamma<\SL_3(\br)$ such that the locally symmetric space $\Gamma\ba X$ contains a sequence
    of floating geodesic planes  whose closures have Hausdorff dimensions strictly bigger than $2$ and accumulating at $2$. Moreover, $\Gamma$ can be chosen inside $\SL_3(\Z)$.
\end{theorem}
The proof of Theorem~\ref{m1} requires a combination of geometric and
dynamical arguments that go well beyond existing rank-one techniques. Substantial difficulties already arise in the \emph{Fuchsian case}: even for a Fuchsian surface group,
relating the Hausdorff dimension of floating geodesic plane closures in the locally symmetric space $\Gamma \backslash X$ to the dynamics of the geodesic flow on the underlying hyperbolic surface requires delicate analysis beyond homogeneous space considerations.

To transport these fractal phenomena from the Fuchsian locus to Zariski-dense Hitchin subgroups, we combine Goldman’s bulging deformations
with a central geometric ingredient:
a detailed analysis of the nearest-point projection to the reference plane.
In higher rank, the fibers of this projection contain entire maximal
flats, creating parallel families of geodesics and making stability under
deformation substantially more subtle than in rank one.

We now explain the bulging deformations.
We begin with a torsion-free cocompact Fuchsian group $\Gamma_0 < H$ and choose a simple closed geodesic $\beta $ in the hyperbolic surface $S=\Gamma_0 \backslash Y$.
This curve $\beta$, called the bulging locus, is represented by some hyperbolic element $\delta\in \Ga_0$. Goldman introduced the notion of bulging deformation along  such a curve $\beta$ (\cite{Go-bulging}; see also a recent work \cite{BHM} where they use the terminology {\it grafting} instead of bulging). Roughly speaking, bulging is an analogue of Thurston's earthquake deformation, but in the setting of convex projective structures: one ``bends'' the geometry of the surface along $\beta$ by a parameter from the  identity component centralizer $C_G(\delta)$ of $\delta$.

Formally, the deformation yields a representation $\rho_{\beta,\b} : \Gamma_0 \to G$ for any $\b\in C_G(\delta)^\circ$, which lies in the {\em Hitchin component} of the character variety of representations of $\Gamma_0$ into $G$. Choi–Goldman identified this Hitchin component with the space of marked convex $\mathbb{R}P^2$-structures on the surface $\Gamma_0\backslash Y$ \cite{CG}: the bulging deformation then corresponds to varying the convex projective structure by stretching along $\beta$.  In particular, each representation $\rho_{\beta, \b}$ is discrete and faithful, 

 We therefore obtain a discrete subgroup of $G$:
\[
 \Gamma_{\beta, \b} \coloneqq \rho_{\beta,\b}(\Gamma_0);
\] see \eqref{hit} for further details. Moreover $\Ga_{\beta,\b}$ is Zariski dense whenever the width $\wdt(\b)$ of $\mathsf b$ is nonzero (see \eqref{bs} for the definition). Subsequent developments have provided broader frameworks for understanding these groups: Labourie \cite{La} introduced the notion of {\em Anosov representations}, while Fock–Goncharov \cite{FG} developed the theory of {\em positive representations}.
Both have become central tools in the study of Hitchin representations of surface groups into split semisimple real Lie groups.

\begin{figure} 
  \includegraphics [height=6cm]{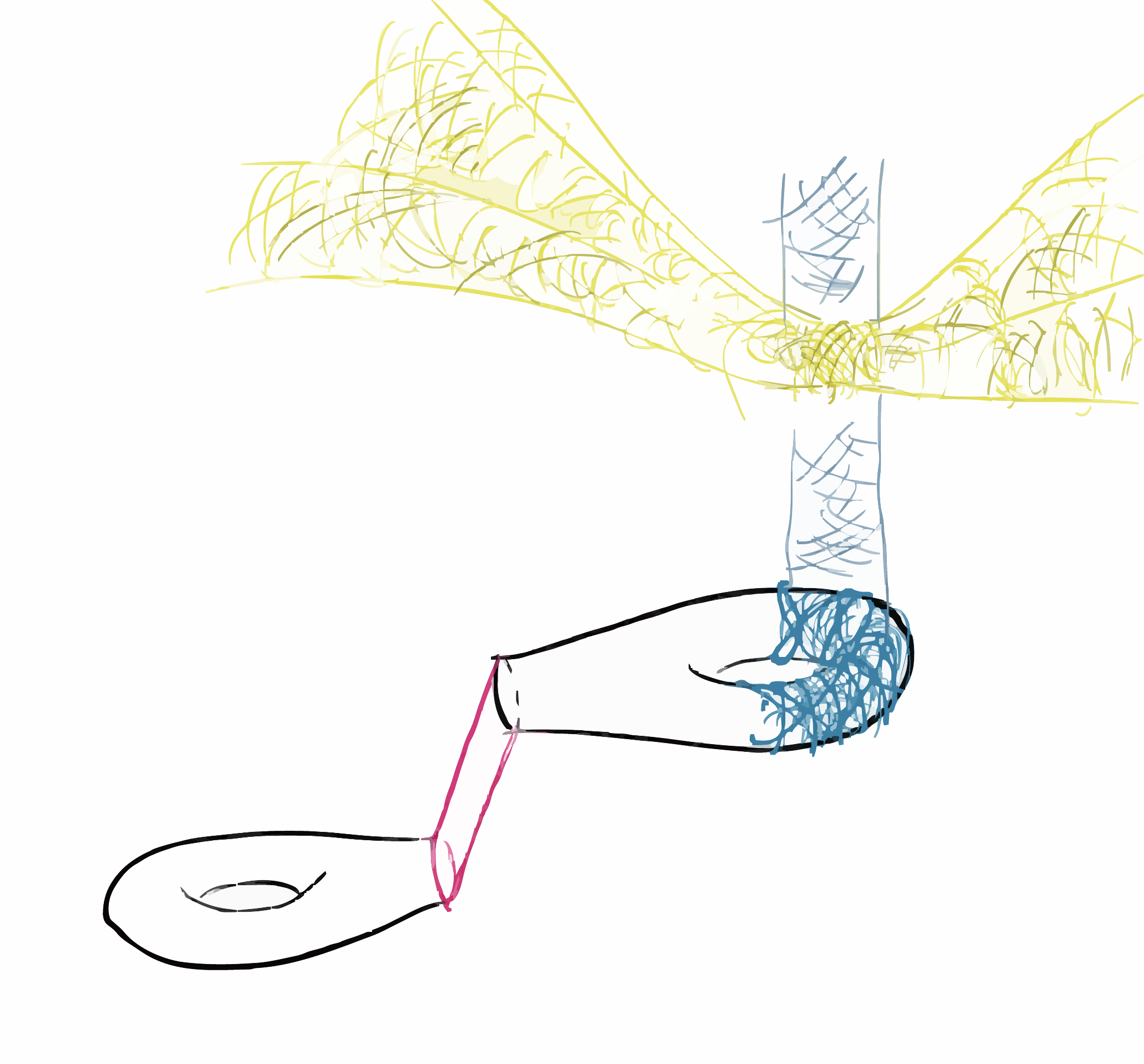}  
\caption{Fractal closures of floating geodesic planes}
\end{figure}

\medskip 
Theorem \ref{m1} is deduced from the following result,
which shows that floating planes inherit their Hausdorff dimensions from the geodesic dynamics on the underlying hyperbolic surface $S=\Ga_0\ba Y$. In this paper,
the notation $\dim $ always refers to the Hausdorff dimension.
\begin{theorem} \label{ma}
Let $\G_0<H$ be a torsion-free cocompact Fuchsian subgroup, and let $S=\Ga_0\ba Y$.
Let $\beta\subset S$ be a simple closed geodesic represented by some $\delta\in \Ga_0$.

Let $ L\subset  Y$ be an admissible geodesic (Def. \ref{adm}), and set $\ell =\Ga_0\ba \Ga_0 L\subset S$. 
Suppose that $1<\dim (\overline{\ell}) <2$ and let $r:=d(\ell , \beta)>0$.
Then there exists $t_0>0$, depending only on $L$ and $r$, such that for all $t>t_0$ and any $\b\in C_G(\delta)^\circ$ with width $\wdt(\b)$ smaller than $r/2$, the Hausdorff dimension of the closure of the floating geodesic plane in $\Ga_{\beta, \b } \ba X$ satisfies  
 $$\tfrac{1}{2}\left( \dim \overline{\ell} +3 \right) \le  \dim \left( \overline{ \Ga_{\beta,\b } \ba \Ga_{\beta, \b } Y_{L,t}} \right) \le \dim \overline{\ell} +1 .$$

Moreover, there exists a sequence of admissible geodesics $L_i \subset Y$ such that 
\begin{itemize}
    \item  $r = \inf_i d(\ell_i,\beta) >0$, 
    \item $\dim \bar\ell_i >1$ for all $i\in\N$, and
\item $\dim \bar\ell_i \to 1$ as $i\to \infty$. 
\end{itemize}

Hence for all sufficiently large $t$,
$$ \dim \left( \overline{ \Ga_{\beta,\b } \ba \Ga_{\beta, \b } Y_{L_i,t}} \right) \to 2 \quad\text{as $i\to \infty$} .$$
\end{theorem}

\begin{Rmk}
 Chaotic behavior can also be produced for 
    geodesic planes orthogonal to a fixed irreducible geodesic plane, via their intersection locus (see Theorem \ref{ort}). Nevertheless, we find the geometry of the floating geodesic planes both more compelling and novel, and hence focus on them in this paper.
\end{Rmk}

 For a concrete analysis, we realize $H$ as the identity component of the special orthogonal group associated to the quadratic form $F(x,y,z)=2xz-y^2$. Let $A<G$ be the subgroup of positive diagonal matrices and set $A_0=H\cap A$.
 We can identify $H$ with the unit tangent bundle of $Y$, and the right translation action of 
$A_0$ on $H$ with the geodesic flow. Every geodesic $L\subset Y$ is of the form $h A_0 o$ for some $h\in H$.
For \[
a_t = \diag (e^t, e^{-2t}, e^t),
\]
the floating geodesic plane over $L$ at height $t$ is
$$Y_{L, t}=ha_t H(o). $$
Theorem \ref{ma} is deduced from the following statement about the  closure of the projection of $Y_{L, t}$ in $\Ga_{\beta, \b}\ba G$:
\begin{theorem} \label{ma3} Under the same hypotheses as Theorem \ref{ma}, we have
$$\dim \overline{ \Ga_{\beta, \b} \ba \Ga_{\beta, \b} h a_t H }=\dim \overline{\Gamma_0\ba \Gamma_0 h A_0} +2 ;$$ 
$$ \tfrac{1}{2}\left(  {3+  \dim \overline{\Gamma_0 hA_0} }\right)  \le  \op{dim} \overline{ \Ga_{\beta, \b} Y_{L, t} }\le  1+ \op{dim} \overline{\Gamma_0 h A_0}  .$$

Moreover, there exists a sequence $h_i\in H$ such that
$ \dim \overline{ \Ga_{0} \ba \Ga_{0} h_i A_0 }$ accumulates at $ 1$ as  $i\to \infty$.
 \end{theorem}

\subsection*{On the proof ideas}

Roughly speaking, the proofs of Theorems~\ref{ma} and~\ref{ma3} proceed in four steps.
A guiding principle is that the essential fractal phenomena already appear in the
undeformed Fuchsian setting, while additional geometric analysis is required to
transport these phenomena through bulging deformations.

\subsubsection*{1. The Fuchsian case and ambient dimension}

We begin with the undeformed (Fuchsian) setting.
Consider the decomposition
\[
H = A_0 k_0 A_0 K_0,
\]
where $k_0\in K_0$ is the quarter turn, $A_0 = H\cap A$, and $K_0 = H\cap K$.
For a cocompact Fuchsian subgroup $\Gamma_0 < H$, we show that if the orbit
$\Gamma_0 \backslash \Gamma_0 h A_0$ is admissible (see Theorem~\ref{dim}), then
\[
\overline{\Gamma_0 \backslash \Gamma_0 h a_t H}
=
\overline{\Gamma_0 \backslash \Gamma_0 h A_0}\, a_t k_0 A_0 K_0,
\]
and moreover
\[
\frac12 \bigl( \dim(\overline{\Gamma_0 \backslash \Gamma_0 h A_0}) + 3 \bigr)
\;\le\;
\dim \overline{\Gamma_0 \backslash \Gamma_0 h a_t H(o)}
\;\le\;
\dim(\overline{\Gamma_0 \backslash \Gamma_0 h A_0}) + 1 .
\]
This provides an explicit description of orbit closures and their Hausdorff
dimensions in the Fuchsian case. The main difficulty is dimension-theoretic and
arises when passing from $\Gamma_0\backslash G$ to $\Gamma_0\backslash X$:
while the natural product map is locally bi-Lipschitz on $\Gamma_0\backslash G$,
the corresponding parametrization in $\Gamma_0\ba X$ fails to be locally
injective. Consequently, Hausdorff dimension in $\Gamma_0\backslash X$ does not
follow formally from the geodesic-flow closure. To obtain the above Hausdorff
dimension statements, we therefore impose an
admissibility condition on the underlying geodesic orbit, ensuring sufficient
regularity of its local projections. This Fuchsian analysis forms the baseline
for the entire argument.
\subsubsection*{2. Nearest-point projection and higher-rank geometry}

To transport the Fuchsian dimension statements to Zariski-dense deformations,
a central geometric ingredient of the paper is a detailed analysis of the
nearest-point projection
\[
\pi : X \to Y .
\]
We prove that for any complete geodesic $L\subset Y$,
\[
d_{\mathrm{H}}(\pi(Y_{L,t}), L) \longrightarrow 0
\quad \text{as } t\to\infty.
\]
In higher rank, however, the fibers of $\pi$ contain entire maximal flats,
giving rise to parallel families of geodesics. This substantially complicates
the analysis of projection dynamics and boundary behavior.
The argument relies on a careful study of Busemann functions
$\beta_\xi|_Y$ for accumulation points $\xi$ of $Y_{L,t}$ in $\partial_\infty X$,
together with precise control of how sequences approach the visual boundary.

\subsubsection*{3. Stability under bulging}

Let $F_{\mathbf b}:\Gamma_0\backslash X \to \Gamma_{\beta,\mathbf b}\backslash X$
denote the map induced by bulging.
Fix $r_0$ strictly larger than the width $\wdt(\mathbf b)$ of $\mathbf b$, and set
\[
X_{\mathbf b}:=\{x \in X:\ d(\pi(x), \Gamma_0 \beta)\ge r_0\}.
\]
We prove that the restriction of $F_{\mathbf b}$ to
$\Gamma_0 \backslash X_{\mathbf b}$ is a proper local isometric embedding into
$\Gamma_{\beta,\mathbf b} \backslash X$.

The neartest-point projection estimates from Step~2 imply that for large $t$, the projection
of $Y_{L,t}$ remains close to its reference geodesic $L$, and hence uniformly
away from the bulging locus. Thus $\Gamma_0\backslash \Gamma_0 Y_{L,t}$ lies in
$\Gamma_0\backslash X_{\mathbf b}$ for $\mathbf b$ of sufficiently small width. Because fibers of $\pi$ contain maximal flats, bulging can in principle
cause different fibers to overlap. Showing that they remain disjoint
under sufficiently small deformations is therefore a genuinely
higher-rank phenomenon. A careful analysis of how bulging modifies
these fibers yields the required stability and allows us to transfer
the orbit-closure and ambient dimension statements from the Fuchsian
case to all $\mathbf b\in C_G(\delta)^\circ$ with sufficiently small width.

\subsubsection*{4. Fractal geodesic closures away from the bulging locus}

Finally, we construct \emph{admissible} geodesics in closed hyperbolic surfaces whose
closures have Hausdorff dimension arbitrarily close to $1$,
while remaining uniformly separated from the bulging locus.
This construction uses Sullivan’s ergodicity theorem for the
Bowen–Margulis–Sullivan measure on convex cocompact surfaces~\cite{Su}.

Combining this with the existence of Zariski-dense Hitchin subgroups
$\Gamma < \SL_3(\mathbb Z)$ due to Long–Thistlethwaite~\cite{LT},
and applying the above deformation procedure,
produces the arithmetic examples required in Theorem~\ref{ma}.

\medskip 
\noindent{\bf Acknowledgment} We would like to thank Federico Rodriguez Hertz for helpful conversations on geodesic closures on hyperbolic surfaces. We would also like to thank Dongryul Kim for useful comments on the preliminary version of this paper.
S.D. thanks MPI MiS in Leipzig, where part of this work was completed.

\section{Limit sets in \texorpdfstring{$G/P$}{G/P}}
Let $G=\SL_3(\br)$. Let 
\begin{equation}\label{cartan-inv}
    \Theta:G\to G
\end{equation} be the Cartan involution given by $\Theta(g)=(g^T)^{-1}$ for $g\in G$. Let $K<G$ be the maximal compact subgroup:
$$K=\SO(3)=\{g\in G: \Theta(g)=g\}.$$
Let $A<G$ be the diagonal subgroup:
$$A=\{\text{diag}(a_1, a_2, a_3)\in G: a_1, a_2, a_3 >0\} .$$
Let $\frak g=\frak{sl}_3(\br)=\{x\in \op{M}_3(\br): \op{Tr} x=0\}$, $\frak k=\frak{so}_3=\{x\in \frak {sl}_3(\br): x=-x^T\}$ and $\frak a$ denote the Lie algebra of $G$,  $K$ and $A$, respectively.
We may identify $\fa$ with the hyperplane 
$$\fa=\{u=(u_1, u_2, u_3)\in \br^3: u_1+u_2+u_3=0\} .$$
Let $\alpha_1$ and $\alpha_2$ be the simple roots of $(\frak g, \frak a)$ given by
$$\alpha_i(u_1, u_2, u_3)=u_i-u_{i+1}\quad i=1,2.$$
Let $\fa^+$ denote the positive Weyl chamber
$$\fa^+= \{(u_1, u_2, u_3)\in \fa : u_1\ge u_2\ge u_3\}.$$
and set
$$A^+=\exp \fa^+ .$$

The  Killing form on $\frak g$ is 
$$\op{B}(x, y)= 6\op{Tr}(xy), \quad x, y\in \frak{sl}_3(\br),$$
which induces the inner product on $\fa$: for $u, v\in \fa$,
$$\langle u, v \rangle = 6 (u_1v_1+u_2v_2+u_3v_3) .$$
We denote by $\|\cdot \|$ the corresponding norm on $\fa$.

Let $X=G/K$ be the Riemannian symmetric space equipped with the metric $d_X$ induced by the Killing form. It is a nonpositively curved CAT (0)-space \cite{BH}. We also consider a left $G$-invariant and right $K$-invariant metric $d_G$ on $G$ compatible with $d_X$. We will omit subscripts and write $d$ for both metrics.  Let $o=[K]\in {X}$.
We have the Cartan decomposition $$G=KA^+K,$$ 
which says that for any $g\in G$, there exists a unique element $\mu(g)\in \fa^+$ such that
$g\in K\exp \mu(g) \, K$. The map $$\mu: G\to \fa^+$$ is called the Cartan projection.

Let $P$ be the upper triangular subgroup of $G$
$$P=\begin{pmatrix}
    * & * & *\\ 0 & * & *\\ 0 & 0 & *
\end{pmatrix},$$
which is a minimal parabolic subgroup of $G$. Then $P=MAN$ where 
$$M=\{\diag(\epsilon_1, \epsilon_2, \epsilon_3)\in G: \epsilon_i\in \{1,-1\}, i=1,2,3\}$$ is the centralizer of $A$ in $K$
and $N$ is the strictly upper triangular subgroup.

There are two maximal parabolic subgroups of $G$ containing $P$: $$P_1=\begin{pmatrix}
    * & * & *\\ 0 & * & *\\ 0 & * & *
\end{pmatrix} \quad\text{and}\quad  P_2=\begin{pmatrix}
    * & * & *\\ * & * & *\\ 0 & 0 & *
\end{pmatrix}  .$$ 

Denote by $e_1, e_2, e_3$ the standard column vectors of $\br^3$. The group $G$ acts transitively on the projective space $\br P^2$, and the stabilizer of the point $[e_1]\in \br P^2$ is $P_1$. Therefore
 $G/P_1$  can be identified with $\br P^2$.
Similarly, the stabilizer of the line $[e_1\wedge e_2]$ in $\br P^2$ is $P_2$ and $G/P_2$ can be identified with the space of lines in $\br P^2$.

Since $P=P_1\cap P_2$, the map  $gP\mapsto (gP_1, gP_2)$ defines an embedding
\be\label{pp} G/P\hookrightarrow G/P_1\times G/P_2,\ee 
via which we identify $G/P$ with its image. Thus
 $$G/P$$ is the full flag variety in $\br^3$:
a point $\xi\in G/P$ is a pair $(p, \ell)$ where $p\in \R P^2$ is a point and $\ell$ is a line in $\R P^2$ containing $p$.

We now define convergence to points in $G/P$ for sequences in $G$ and $X$.
\begin{Def} \label{dur} Let $g_n\in G$ be a sequence. \begin{itemize}
    \item 
We say 
$g_n \to \infty$ regularly if
$\alpha_i(\mu(g_n))\to \infty$ for both $i=1,2$. 

\item We say  $g_n\to \xi\in G/P$ if $g_n\to \infty$ regularly and $\xi =\lim_{n\to \infty} k_n P$ where $k_n\in G$ is a sequence such that $g_n\in k_n A^+K$.

\item We say $g_n \to \infty$ {\em uniformly regularly} if $g_n\to \infty$ regularly and 
there exists  $c>0$ such that for each $i=1,2$, 
$$\alpha_i(\mu(g_n)) \ge c \|\mu(g_n)\|\quad\text{ for all $n\in \N$.}$$

\item For a sequence $x_n=g_n o\in X$, 
we say $x_n\to \infty$ regularly (resp. uniformly regularly) if $g_n\to \infty $ regularly (resp. uniformly regularly) and that  $x_n\to \xi\in G/P$ if $g_n\to \xi$.
\end{itemize}
\end{Def}

By the Cartan decomposition $G=KA^+K$, these notions are all well-defined.

\begin{Def}[Limit sets]
For $Z\subset G$,  the limit set $\Lambda_Z$ is the set of all accumulation points of sequences from $Z$ in $G/P $. Similarly, if $Z\subset X$, the limit set $\La_Z$
is the set of all accumulation points of sequences from $Z$ in $G/P $.
The projection of $\La_Z$ to $G/P_i$ via \eqref{pp} will be referred to as the limit set of $Z$ on $G/P_i$ for $i=1,2$.
\end{Def}

\begin{lemma}\label{same}
    If $Z_1, Z_2\subset X$ have bounded Hausdorff distance,
    then  $\La_{Z_1}= \La_{Z_2}$ in $G/P$.
\end{lemma}
\begin{proof} Suppose that the Hausdorff distance between $Z_1$ and $Z_2$ is at most $R$. Let $\xi\in \La_{Z_1}$. 
Then some sequence $g_i o\in Z_1$ converges to $\xi$. There exists a sequence
    $q_i\in G$ with $d(q_io, o)\le R$ such that $g_i q_i o\in Z_2$. By \cite[Lemma 2.10]{LO}, $g_iq_i o\to \xi$. Hence $\xi\in \La_{Z_2}$. Reversing the role of $Z_1$ and $Z_2$ gives the claim.
\end{proof}

\section{Nearest projection to \texorpdfstring{$Y$}{Y}}\label{sec:3}
 Fix the quadratic form $F$ in $\br^3$:  $$F\begin{pmatrix}
     x \\ y\\ z
 \end{pmatrix}=2xz-y^2.$$
Let $H$ be the identity component of the special linear group
$$\SO(F)=\left\{g\in G: F\left( g (v)\right)  =F (v)  \text{ for all $v\in \br^3$}\right\} .$$
For the symmetric matrix \be\label{J} J=\begin{pmatrix} 0 & 0 &1 \\
0& -1 & 0 \\ 1& 0 & 0 \end{pmatrix},\ee 
we have
$$H=\{g\in G: g^T J g=J\}^\circ .$$
In particular, $H\simeq \SO(1,2)^\circ$.

The Lie algebra of $H$ is
\be\label{Lieh}\frak h=\left\{ \begin{pmatrix} s & x & 0 \\ y & 0 &x\\ 0 & y & -s
\end{pmatrix}: x, y, s\in \br \right\}.\ee

Let $\cal S$ be the space of all symmetric matrices of signature $(1,2)$ with determinant one. 
The action of $G$ on $\cal S$ by $g.J_0= gJ_0g^T$ identifies
 $G. J\simeq G/\SO(F)\simeq \cal S$. A non-degenerate quadric in $\R P^2$ is a projective circle (ellipses, hyperbolas, paraboloids). 
 The map \be\label{gso} g\SO(F)\mapsto \{F\circ g=0\}\ee  identifies $G/\SO(F)$
with the space of non-degenerate quadrics in $\R P^2$ \cite{Go-bulging}.

For $s\in \br$, set
$$h_s=\begin{pmatrix}e^s &0 & 0\\ 0 & 1& 0 \\ 0 &0&e^{-s}
\end{pmatrix}\in H .$$
Note that $h_s$ is a regular semisimple element, i.e., its centralizer is the diagonal subgroup, and that $h_s\to \infty$ regularly as $|s|\to \infty$. These are important facts which will be  used often.
We set
$$A_0:=H\cap A =\{h_s:s\in \br\} \quad\text{and} \quad K_0:=H\cap K.$$  The Lie algebra of $K_0$ is  $$
  \mathfrak{k}_0
  = \left\{
 \begin{pmatrix}
      0 & \theta  & 0\\ -\theta  & 0 & {\theta} \\ 0 & -{\theta} & 0
    \end{pmatrix}: \theta\in \br 
\right\} .$$
Setting  \be\label{kt}
  k(\theta):=
  \begin{pmatrix}
    \frac{1+\cos\theta}{2} & -\frac{\sqrt2\sin\theta}{2} & \frac{1-\cos\theta}{2}\\[4pt]
    \frac{\sqrt2\sin\theta}{2} & \cos\theta & -\frac{\sqrt2\sin\theta}{2}\\[4pt]
    \frac{1-\cos\theta}{2} & \frac{\sqrt2\sin\theta}{2} & \frac{1+\cos\theta}{2}
  \end{pmatrix},\ee 
 we get the parametrization
 $$K_0 = \{k(\theta) :\theta\in \br\} .$$
The subgroup
$K_0$ is a maximal compact subgroup of $H$ and we have 
\be\label{hak} H=K_0A_0^+K_0\ee 
where $A_0^+=\{h_s:s\ge 0\}$.

\subsection*{The limit set of $Y$} The quotient space
$${Y}:=H/K_0 =K_0A_0o $$ is a totally geodesic subspace of ${X}$. 
The quadric $\{F=0\}=\{y^2 = 2xz\}$ divides $\mathbb{R} {P}^2$ into two $H$-invariant connected components: an open disk \be\label{dis} D = \{[x : y : z]\in \R P^2 :  y^2 < 2xz\}\ee  and an open Möbius band \be \label{mob} \{[x : y : z] \in \R P^2:  y^2 > 2xz\}.\ee

\begin{lem}\label{lem:Lambda_Y} We have $$\Lambda_Y =  K_0 P/P\quad\text{and} $$ 
$$\Lambda_Y =\{ (p,\ell)\in G/P_1\times G/P_2 : p\in \partial D,\ \ell \text{ is tangent to } \partial D \text{ at } p\}.$$
\end{lem}
\begin{proof} Since $h_{s}\to \infty$ regularly as $|s|\to \infty$,
any unbounded sequence in $ Y$ has a subsequence which converges to a point in $G/P$. Moreover, in view of \eqref{hak},
any limit of an infinite sequence in $H $ in $G/P$  belongs to $K_0P/P$. Hence 
$\Lambda_Y = K_0 P/P$. Note that $P$ corresponds to the pair $(p_0, \ell_0) \in G/P_1 \times G/P_2$, where $p_0 = [e_1]$ and $\ell_0 = [e_1\wedge e_2]= \{[x:y:z] :\ z = 0\}$. Moreover,
$\ell_0\cap \partial D=\{p_0\}$ and hence $\ell_0$ is tangent to $\partial D$ at $p_0$.
Since $\partial D$ is a single $K_0$-orbit,  the claim follows.
\end{proof}

We fix the element \be\label{ko} k_0:=\begin{pmatrix}
    \frac{1}{2} & \frac{1}{\sqrt 2} & \frac12 \\ -\frac{1}{\sqrt 2} &0 & \frac{1}{\sqrt 2} \\ \frac12 & -\frac{1}{\sqrt 2} & \frac12
\end{pmatrix}  \in  K_0 .\ee 
The geodesics $k_0A_0 o $ and $A_0o$ in $Y$ are orthogonal. Therefore $Y$ is swept out by the family of orthogonal geodesics to $A_0o$; indeed,
$$Y= A_0 k_0 A_0 o .$$
It follows that
$$H = A_0 k_0 A_0 K_0 . $$

 \subsection*{Generalized Cartan decomposition} 
Recalling the symmetric matrix $J$ from \eqref{J}, consider the following involution $\sigma: G\to G$:
$$\sigma(g)= J \Theta (g) J.$$ 
We then have
$$H=\{g\in G: \sigma(g)=g\}^\circ;$$
therefore $H$ is an affine symmetric subgroup of $G$. The generalized Cartan decomposition of $G$ with respect to $H$ is described in \cite{Sc}, as we recall below.

 Observe that the differential $d\sigma:\frak g\to \frak g $ commutes with $d\Theta$,
 and we have
$$\frak g=\frak k\oplus \frak p=\frak h\oplus \frak q$$
which are decompositions into $\pm$ eigenspaces for $d\Theta$ and $\sigma$, respectively. The subspace  \be\label{LieB}\frak b:=\left\{ \begin{pmatrix}
    t & 0 & s\\ 0 &-2t& 0\\
    s&0&t
\end{pmatrix}: t,s\in \br \right\} \ee
is a maximal abelian subalgebra of $\frak p\cap \frak q$. It is also a maximal abelian subalgebra of of $\frak p$, since the rank of $G$ is $2$. The maximal split torus $B:=\exp \frak b$ of $G$
is
  $$B =\left\{\begin{pmatrix}
    e^t\cosh s & 0 & e^t \sinh s \\
0& e^{-2t} & 0 \\ e^t \sinh s & 0 & e^t \cosh s \end{pmatrix}:t, s\in \br \right\} .$$

We set 
\be\label{k1} {k}_1= \begin{pmatrix}
    \frac{1}{\sqrt 2} &0 &  \frac{1}{\sqrt 2} \\ 0&  -1 & 0 \\  \frac{1}{\sqrt 2} & 0& - \frac{1}{\sqrt 2}
\end{pmatrix} \in K.\ee 
We then have
$$\frak b =k_1 \frak a k_1^{-1}\quad \text{and }\quad  B =k_1 A k_1^{-1}. $$

Letting $\cal W_\sigma=N_K(\frak b)/C_K(\frak b)$ and $\cal W_{\sigma, \theta}= N_{K_0} (\frak b)/C_{K_0} (\frak b)$, choose a finite subset \be\label{W} \cal W\subset N_K(\frak b) \ee of representatives 
for $\cal W_{\sigma, \theta}\ba \cal W_{\sigma} $. Set  $B^+=k_1A^+ k_1^{-1}$. 

\begin{theorem}\label{affine} We have $$G=H BK =H \cal W B^+ K$$
in the sense that for any $g\in G$, there exist unique elements $b\in B^+$ and $w\in \cal W$ such that
$$g\in H wb K.$$
\end{theorem}

\subsection*{Nearest point projection to $Y$}
Since $X=G/K$ is non-positively curved and $Y\subset X$ is totally geodesic,  for any $x\in X$, there exists a unique $y\in Y$ such that $$d(x, y)=\inf_{y'\in Y} d(x, y').$$ This $y$ is called
the {\em nearest point projection} of $x$ to $Y$. See \cite[p. 8]{BGS} for further details.

Let \be\label{pi} \pi:X\to Y\ee  be the nearest projection map.
Since $X=H Bo$, by Theorem \ref{affine}, we have the following description of $\pi$:

\begin{prop}\label{fiber} For any $h\in H $ and $b\in B$,
\be\label{eqn:proj} \pi(hbo)=ho. \ee 
    \end{prop}
    \begin{proof}
    Since $\pi$  is $H$-equivariant, it suffices to show $\pi(z)=o$ for all $z\in Bo$. Since $\op{Tr}(xy)=0$ for all $x\in \frak h$ and $y\in \frak b$, the subspaces $\frak h$ and $\frak b$ are orthogonal to each other. Moreover, $Bo$ and $Y$ are totally geodesic. Thus, for $z\in Bo$, the geodesic segment connecting $z$ to $o$ lies in $Bo$ and is orthogonal to $Y$,  $\pi(z) = o$.
    \end{proof}

\begin{cor}\label{fiber2}
    We have $$\pi^{-1}(o)= K_0  Bo=\bigcup_{w\in \cal W} K_0 {k}_1 w A^+ o$$
    and
$$  \La_{\pi^{-1}(o)}= \bigcup_{w\in \cal W} K_0 {k}_1 w P$$
    where ${k}_1$ and $\cal W$ are given in \eqref{k1} and \eqref{W} respectively.
\end{cor}
\begin{proof}
    The first part follows from \Cref{fiber} and  \Cref{affine}. The second follows from the definition of the limit set and the fact that $K_0{k}_1 \cal W \subset K$.
\end{proof}

\section{Floating geodesic planes}
As in the last section, consider the totally geodesic plane $$ Y=H o \subset X.$$

Since the centralizer of $A_0$ is equal to the diagonal subgroup of $G$, it follows that $Ao$ is the unique maximal flat in $X$ containing the geodesic $A_0o$. 
  \begin{lem}
      The geodesic plane $Y$ is perpendicular to the flat $Ao$ and $Y\cap Ao=A_0 o$.
  \end{lem}
  \begin{proof}
      The intersection of $Y$ and $Ao$ is a totally geodesic submanifold of $X$, which contains $A_0o$. Since neither $Y$ nor $Ao$ contains the other, dimensional considerations imply that the intersection is precisely $A_0o$.

      The Lie algebra $\mathfrak{a}$ splits orthogonally as a direct sum  $\mathfrak{a}_0\oplus \mathfrak{a}'$ where $\fa_0=\op{Lie}(A_0)$ and $\fa'=\{\diag(t, -2t, t): t\in \br\}$. Since $\mathfrak{a}_0 \subset \mathfrak{h}$  and $\mathfrak{a}' \perp \mathfrak{h}$,
      the claim about orthogonality thus follows.
  \end{proof}

For $t\in \br$, set $$a_t=\begin{pmatrix}
    e^t & 0 & 0 \\ 0 &e^{-2t}  & 0 \\ 0 & 0 & e^t
\end{pmatrix} .$$

\begin{Def} Given a complete geodesic $L $ in ${Y}$ and $t\in \br$,
 define the {\em floating} geodesic plane 
$${Y}_{L, t}:= h a_t Y ,$$
where $h\in H$ is chosen so that $L= h A_0 o $; this is well-defined since $h$ is unique modulo the action of $A_0$, which commutes with $a_t$. 
\end{Def}
For the geodesic $L=A_0o$, we simply write $$Y_t:=Y_{L, t}=a_tY.$$

\begin{lem} \label{LL} For $L=hA_0 o$ and $t\in \br$, let $L_t:= h a_t A_0 o\subset Y_{L, t}$.
We have $\pi(L_t)=L$
and $d(h h_r o, h a_t h_r o)=|t|=d(L, L_t) $ for all $r\in \br$.
    \end{lem}
    \begin{proof} By the $H$-equivariance of $\pi$, it suffices to consider the case when $L=A_0 o$.  Since $a_t A_0=A_0 a_t$ and $\pi(a_t o)=o$,
    we get $\pi(a_t A_0 o)=A_0 \pi(a_t o)= A_0 o$.
Similarly,
$d(h_r o, a_t h_ro)=d(h_ro, h_r a_t o)=d(o, a_to)=|t|.$ Since $L$ and $L_t$ lies in the same flat $Ao$, and $L$ and $\{a_to:t\in \br\}$ are orthogonal at $o$, $d(L, L_t)=d(o, a_t o)=|t|$.
    \end{proof}

The geodesic plane ${Y}_{L, t}$
is ``ultra-parallel'' to $Y$ at distance $t$:

\begin{lemma}\label{min}
    Given a geodesic $L=hA_0o \subset Y$ and $t\in \R$, we have 
    $$d(Y, Y_{L, t})=\min \{ d(y,z) :\ y\in Y,\, z\in {Y}_{L, t}\} = |t|.$$ Moreover, for $t\ne 0$, the locus where the distance is minimized is precisely $\{ (hh_r o,h a_t h_r o) :\ hh_ro \in L\}$.
\end{lemma}

\begin{proof} By the $H$-equivariance, it suffices to prove the claim when $L=A_0o$ and $Y_{L,t}=a_t Y$.
    Let $y \in Y$ and $z\in {Y}_{ t}$.
    Let $y_0 \in L$ (resp. $z_0 \in L_t=a_t L$) denote the nearest point projection of $y$  (resp. $z$) to $L$ (resp. $a_tL$). Let $L_0$ be the geodesic segment connecting $y_0$ to $z_0$. Since the flat $Ao$ is orthogonal to $Y$ and $a_t\in A$, $Ao$ is also orthogonal to  ${Y}_{t}$. It follows that the points $y$ and $z$ lie on two geodesics perpendicular to $L_0$ and passing through its endpoints.
    
By \cite[Ch. I, Prop. 5.4]{Ballmann},
     in a non-positively curved space, 
   for any geodesic segment $[a,b]$ and for any perpendicular complete geodesics $L_1(t)$ and $L_2(t)$ to $[a,b]$ with $L_1(0) = a$ and $L_2(0) = b$, we have
   $$\inf\{ d(L_1(t_1),L_2(t_2)) :\ t_1\in\R,t_2\in\R\} = d(a,b).$$
Since the geodesic segment $[y, y_0]$ lies in $Y$ and is perpendicular to $L$, it is perpendicular to the whole maximal flat $Ao$. Similarly, 
 $[z, z_0]$ is also perpendicular to  $Ao$.
   So 
    \be\label{above}
     |t| \le d(y_0, z_0) \le d(y,z).
    \ee 
This proves the first claim.

  For the second claim, without loss of generality, assume that $t>0$ and
    suppose that  $d(y,z) = t$. It suffices to show that $y=y_0$, $z=z_0$ and $z_0=a_ty_0$.
  The inequality \eqref{above} forces $t= d(y_0, z_0)$. Hence by Lemma \ref{LL},
  $d(L,a_tL)=t=d(y_0,z_0)$. Therefore  the geodesic segment $[y_0,z_0]$ is perpendicular to $L$ and $a_t L$.  Since $a_t$ translates $Ao$ orthogonal to $L$, it implies that  $z_0 = a_t y_0$.

    We now claim that 
    \be\label{yoo} \text{ $y=y_0\quad $ and $\quad z=z_0$. }\ee 
    First suppose that $\{y_0, z_0\}\cap 
    \{y, z\}=\emptyset$. 
    Two complete geodesics in $X$ are either parallel
    (they have finite Hausdorff distance) or  the minimum distance
    between them is realized by a unique pair of points or the minimum distance is not realized. (This follows from \cite[Ch. I, Prop. 5.4]{Ballmann}.)
    Hence in the setting at hand, 
    the complete geodesics $\cal G_1$ and $\cal G_2$ in $X$ passing through $y, y_0$ and $z,z_0$ respectively must have a finite Hausdorff distance.
 Since $A_0 \cal G_1= Y $ and
    $A_0 \cal G_2= Y_t $,
 ${Y}$ and $Y_t$ are at a finite Hausdorff distance. Thus, the limit sets of $Y$ and ${Y}_{t}$ in $G/P_1=\mathbb R P^2$ are the same by Lemma \ref{same}. Since $G/\SO(F)$ is in bijection with the space of non-degenerate quadrics  via the map in \eqref{gso} it follows that $Y_{ t}=Y$. Hence $t=0$, a contradiction.

    Now suppose that $\{y_0, z_0\}\cap \{y, z\}$ is a singleton. Without loss of generality, we may assume that $y=y_0$ and $z\ne z_0$.
    Since the distance function is strictly convex in a Hadamard manifold \cite[Sec. 1.4]{BGS}, any point $z_1$ in the geodesic segment $[z, z_0]$ other than the endpoints satisfies $d(y, z_1)<d(y,z)=t$, which is a contradiction to $t=d(\cal G_1, \cal G_2)$.
    Therefore $\{y_0, z_0\}=\{y, z\}$, and consequently 
    $y_0=y$ and $z_0=z$, proving \eqref{yoo}.
    \end{proof}

  \begin{Rmk}\label{tg}
     Up to an isometry, any totally geodesic plane in $ X$ is given by
    $Y$, or  $ \SL_2(\br) o$, where $\SL_2(\br)$ is embedded as the left upper corner
        of $\SL_3(\br)$, or a maximal flat
    $A(o)$. It is natural to call first type of geodesic planes as {\it irreducible} geodesic planes.
     \end{Rmk}

\subsection*{The limit set of the floating planes}

For each $t\in \br$, consider the quadric 
$$Q_{t}= \{[x:y:z]\in \mathbb R P^2: e^{4t} y^2=2 e^{-2t} xz\}$$ passing through $[e_1]$ and $[e_3]$. This is a projective circle.

Since $a_t$ sends the boundary of the disk $D= \{[x : y : z]\in \R P^2 :  y^2 < 2xz\}$ to $Q_t$,  Lemma \ref{lem:Lambda_Y}
implies: \begin{lem}
   For $t\in \br$, the limit set of $a_t Y$ in $G/P$ is given by 
    $$\La_{a_tY}=\{(p, \ell)\in G/P: p\in Q_t, \ell \text{ is a line tangent to $Q_t$ at $p$}\} .$$
\end{lem}

\section{Limits of the sequence \texorpdfstring{$\ga_t(s)=a_tk_0h_s o$}{gamma(t)(s) = a(t) k(0) h(s) o}}
Since $Y=A_0 k_0 A_0 o$, where  $ k_0\in K_0 $ is as in \eqref{ko}, and the nearest projection map $\pi$ defined in \eqref{pi} is $H$-equivariant, we have
$$\pi(a_t Y)=A_0 \pi (a_t k_0 A_0 o)\quad\text{ for any $t\in \br$.} $$ Therefore to understand the image $\pi(a_t Y)$, it suffices to analyze the sequence
 \be\label{gt} \ga_t(s):=a_tk_0h_s o . \ee
 In this section we determine all accumulation points of
 $\ga_t(s)$ in $G/P$, according to the relative rates at which $t$ and $s$ tend to $\infty$.

The main goal in this section is to show:
\begin{prop}\label{every0}
    Any accumulation of the sequence $\ga_t(s)$ in $G/P$ as $t, |s|\to \infty$ belongs to
    $\La_{\pi^{-1}(o)}$. Moreover, we have  $\liminf |s_n|/t_n>0$ if and only if $\ga_{t_n}(s_n)\to \infty$ uniformly regularly.
\end{prop}
We begin with calculating the Cartan projection of such a sequence up to a uniform bounded subset:
\begin{lem} \label{cr} There is a compact subset $C\subset \fa$ such that 
for any $t>0$ and $ s\in \br$, the Cartan projection $\mu(a_t k_0h_s) $  satisfies 
$$\mu(a_t k_0 h_s) \in  \begin{pmatrix} t+|s| & 0 &0\\ 0 & t &0\\ 0& 0& -2t-|s|
\end{pmatrix} + C .$$
    \end{lem}
    \begin{proof}
 For simplicity, set $g=a_tk_0 h_s $. 
 Write $g=k a \ell \in KA^+K $ in Cartan decomposition  so that $\mu(g)=\log a$, where $a=\diag (a_1, a_2, a_3)$.
We estimate the $a_i$ up to a uniform multiplicative constant.

Since  $gg^T=k a^2 k^{-1}$, the eigenvalues of $gg^T$ determine the $a_i^2$.
Let  $c_s=e^s+e^{-s}$ and $d_s= e^s-e^{-s}$.
A direct computation gives 
$$ g g^T=\frac{1}{4} \begin{pmatrix} e^{2t}c_s^2 & -\sqrt 2 e^{-t}c_sd_s & e^{2t}d_s^2 \\
-\sqrt 2 e^{-t} c_sd_s & 2 e^{-4t} (e^{2s}+e^{-2s}) & -\sqrt 2 e^{-t}c_sd_s\\ e^{2t} d_s^2 &-\sqrt 2 e^{-t}c_sd_s & e^{2t}c_s^2
\end{pmatrix} .$$

Since $t>0$ and $|d_s|\le c_s$, the Frobenius norm of $gg^T$ satisfies  
$$   \|gg^T\|^2\asymp     e^{4t} c_s^4 .$$
Here, $\asymp$ denotes equality up to a uniform multiplicative constant.
As $ \|gg^T\|^2= a_1^4$ and $  c_s\asymp   e^{|s|} $,
we get
\be\label{c1}  a_1 \asymp  e^t e^{|s|}. \ee 

For $a_2$, we now
 compute $\wedge^2 (gg^T)$ with respect to the ordered basis $e_2\wedge e_3, e_1\wedge e_3, e_1\wedge e_2$: 
$$\wedge^2(gg^T)=\frac{1}{4} \begin{pmatrix}  e^{-2t}c_s^2  & -\sqrt 2 e^t c_sd_s & e^{-2t} d_s^2 \\
-\sqrt 2 e^t c_sd_s    &  2e^{4t}(e^{2s}+e^{-2s})  &  -\sqrt 2 e^t c_sd_s \\
 e^{-2t}d_s^2 &   -\sqrt{2} e^t c_sd_s  &   e^{-2t} c_s^2
\end{pmatrix}.$$

Hence
$$  \| \wedge^2(gg^T)\|^2\asymp   e^{8t} e^{4|s|} .$$

Since the exponential of the Cartan projection of $\wedge^2(gg^T)$ has entries
$a_1^2a_2^2, a_2^2a_3^2, a_1^2 a_3^2$ with the largest one being $a_1^2a_2^2$,
we have  $$  \| \wedge^2(gg^T)\|^2 \asymp  a_1^4 a_2^4=   e^{8t} e^{4|s|},$$ which with \eqref{c1} implies 
$$ a_2\asymp    e^{t} .$$
This proves the claim.
\end{proof}

\subsection*{A regularity criterion}

A consequence of \Cref{cr} is as follows:

\begin{corollary}\label{cor:regular}
 Let $\ga_{t_n}(s_n)= a_{t_n} k_0 h_{s_n} o$ with $t_n>0$ and $s_n\in \br$. Then we have
    \begin{itemize}
        \item $\ga_{t_n}(s_n) \to \infty$ regularly if and only if $|s_n|\to \infty$;
  \item 
 $\ga_{t_n}(s_n) \to \infty $ uniformly regularly  if and only if $|s_n|\to \infty$ and 
 $$\liminf_n |s_n|/t_n >0.$$
      \end{itemize}
      
\end{corollary}

Recall that $G/P=KP/P$, identified with 
$$\{ ([ k e_1], [k (e_1\wedge e_2)]): k\in K\}\subset G/P_1\times G/P_2.$$

Proposition \ref{every0} follows from the following together with Corollary \ref{cor:regular}:
\begin{prop}[Limits of $\ga_{t}(s)$]  \label{every}
If  $t_n\to \infty$ and $|s_n| \to \infty$, then  any limit of $\ga_{t_n}(s_n)$ in $G/P$
  belongs to $K_0 {k}_1 P/P$, where $k_1$ is as in \eqref{k1}.
In particular, if $\zeta=k^*P$ is such a limit for some $k^*\in K$, then $k^*e_1$ is proportional to $e_1+e_3$ and
$$\zeta\in  \La_{\pi^{-1} (o)}.$$
\end{prop}

\begin{proof} 
   Write $g_n=a_{t_n} k_0h_{s_n} =k_n  a_n l_n\in KA^+K$ and $ a_n=\text{diag}(a_{n,1}, a_{n,2}, a_{n, 3}).$    By Lemma \ref{cr}, as $|s_n|\to \infty$, we have $ xg_n \to \infty$ regularly.
Suppose  $k_n \to k^*$. 
Using the notation $k(\theta)$ from \eqref{kt},
in order to prove that $k^* \in K_0 k_1 P=\{k(\theta) k_1: \theta\in \br\}P $, it suffices to show that the first two columns of $k^*$ are proportional to $(1,0,1)^T$ and $(-\sin \theta , \sqrt{2} \cos \theta, \sin \theta)^T $ for some $\theta\in \br$.

   Let $$w= \begin{pmatrix} 1&0 & 0\\ 0 & 0 & 1 \\ 0 &-1& 0
\end{pmatrix}  .$$

Then $w^T g_n g_n^T w= w^Tk_n  a_n^2 (w^Tk_n)^T$, so the columns of $w^Tk_n$ are eigenvectors of $w^Tg_n g_n^Tw$ in decreasing order of eigenvalue.
For each $i=1,2,3$, write the $i$-th column vector $$u_{n,i}=w^Tk_n e_i= x_{n, i}e_1+y_{n,i}e_2+z_{n,i}e_3 .$$

We show that the first column of $w^Tk^*$ is parallel to $(-1,1,0)^T$ and the second column  of $w^Tk^*$ is parallel to 
$(p,p, q)^T$ for some $p, q\in \br$. This implies the required structure of $k^*$.

Since $u_{n, i}$ are unit vectors, all $|x_{n,i}|, |y_{n,i}|, |z_{n,i}|$ are at most $1$.
Let  $c_n=e^{s_n}+e^{-s_n}$ and $d_n=e^{s_n}-e^{-s_n}$.
A direct computation gives 
$$Q_n:=w^T g_ng_n^T w =\frac{1}{4} \begin{pmatrix} e^{2t_n}c_n^2  & -e^{2t_n}d_n^2 & -\sqrt 2 e^{-t_n}c_sd_n \\
-e^{2t_n}d_n^2 & e^{2t_n}c_n^2 &  \sqrt 2 e^{-t_n}c_nd_n
\\ - \sqrt 2 e^{-t_n}c_nd_n &  \sqrt 2 e^{-t_n}c_nd_n
& 2 e^{-4t_n} (e^{2s_n}+e^{-2s_n}) 
\end{pmatrix}$$
Setting $f_n=d_n/c_n$, we compute
\begin{align*}  Q_n u_{n,1} &= (w^Tk_n)a_n^2 (w^Tk_n)^T u_{n,1}  = (w^Tk_n)a_n^2 e_1 \\ & =
a_{n,1}^2 u_{n,1}= a_{n,1}^2 (x_{n,1}e_1+y_{n,1}e_2+z_{n,1} e_3)
\end{align*}
On the other hand,
\begin{align*} & Q_n u_{n,1} = x_{n,1}Q_n e_1 +y_{n,1}Q_n e_2+z_{n,1} Q_n e_3 =\\
& \tfrac{e^{2t_n} c_n^2}{4} \left(x_{n,1}  \begin{pmatrix}1\\ - f_n^2\\ \tfrac{ -\sqrt 2f_n}{e^{3t_n}}
\end{pmatrix}  + y_{n,1}   \begin{pmatrix} -f_n^2 \\ 1\\ \tfrac{ \sqrt 2f_n}{e^{3t_n}}
\end{pmatrix}  +z_{n,1} \begin{pmatrix} \tfrac{ -\sqrt 2f_n}{e^{3t_n}} \\ \tfrac{ \sqrt 2f_n}{e^{3t_n}}  \\ \tfrac{2 ({e^{2s_n}+e^{-2s_n}})}{e^{6t_n} c_n^2}  
\end{pmatrix} \right) .\end{align*}

Hence
\begin{multline*}  \frac{4 a_{n,1}^2 }{e^{2t_n} c_n^2}\left( x_{n,1}e_1+y_{n,1}e_2+z_{n,1} e_3\right)=\\
x_{n,1}  \begin{pmatrix}1\\ - f_n^2\\ \tfrac{ -\sqrt 2f_n}{e^{3t_n}}
\end{pmatrix}  + y_{n,1}   \begin{pmatrix} -f_n^2 \\ 1\\ \tfrac{ \sqrt 2f_n}{e^{3t_n}}
\end{pmatrix}  +z_{n,1} \begin{pmatrix} \tfrac{ -\sqrt 2f_n}{e^{3t_n}} \\ \tfrac{ \sqrt 2f_n}{e^{3t_n}}  \\ \tfrac{2 ({e^{2s_n}+e^{-2s_n}})}{e^{6t_n} c_n^2}  
\end{pmatrix}
\end{multline*}
By Lemma \ref{cr},  we may assume 
$r_0:=\lim_{t,s\to +\infty} \frac{4 a_{n,1}^2}{ e^{2t_n}c_n^2} >0$ exists, after passing to a subsequence. Indeed,
$$r_0=2 $$
because $a_{n,1}^2$ is the largest eigenvalue of $w^Tg_ng_n^Tw$. 
Since $\lim_{n\to \infty} f_n= 1$,
taking the limit of the above equation yields
\begin{multline*}
    2 (\lim x_{n,1} e_1+ \lim y_{n,1}e_2+ \lim z_{n,1} e_3)  \\ =\lim x_{n,1}  \begin{pmatrix}1\\ -1\\ 0
\end{pmatrix}  + \lim y_{n,1}   \begin{pmatrix} -1 \\ 1\\ 0
\end{pmatrix}  +\lim z_{n,1} \begin{pmatrix}0\\0\\ 0
\end{pmatrix} 
.\end{multline*}
Comparing the  $e_3$-components gives 
$$\lim z_{n,1}=0.$$
Comparing the  $e_1$ and $e_2$-components gives 
$$\lim y_{n,1}=- \lim x_{n,1}.$$
Thus  $\lim u_{n,1}$ is parallel to $(1, -1, 0)^T$,
that is, the first column of $w^Tk^*$ is parallel to $(-1,1,0)^T$.
 Since each $u_{n,2}$ is orthogonal to $u_{n,1}$, the limit $\lim u_{n,2}$ must be orthogonal to $(-1,1,0)^T$ and hence of the form $(p,p, q)^T$. This proves the claim about the second column of $w^Tk^*$.
\end{proof}

\section{Nearest projection of floating planes to \texorpdfstring{$Y$}{Y}}

Let $\pi:X\to Y$ be the nearest projection map. The main result of this section is as follows:
Fix a complete geodesic $L\subset Y$ and $t\in \br $. Let $Y_{L, t}$ be the associated floating  plane. Then:
\begin{theorem}\label{hd}
The Hausdorff distance between
$\pi( Y_{L,t})$ and $L$ tends to $0$ as $|t|\to \infty$.
\end{theorem}

By the $H$-equivariance of $\pi$, we may assume without loss of generality that 
$$L=A_0(o)\quad\text{ and hence } {Y}_{L, t}=a_t Y.$$

For simplicity, we set $Y_t:=Y_{A_0o, t}$.

\subsection*{Busemann functions}
The visual boundary $\partial_\infty X$ consists of equivalence classes of asymptotic geodesic rays. (Recall that two geodesic rays in $X$ are {\em asymptotic} if they are within a finite  Hausdorff distance). We equip the visual boundary with the {\em cone topology}.

\begin{Def} For $\xi\in \partial_\infty X$,
 the {\em Busemann function} $b_{\xi}:X\to \R$ is 
\[
 b_{\xi} (x) = \lim_{t\to\infty} \left( d(x,\xi_t) - t\right)
\]
where $\{\xi_t: t\ge 0\}$ is the unit speed geodesic ray in the class $\xi$ such that $\xi_0=o$. Since $X$ is nonpositively curved, this is well-defined: there exists a unique unit speed geodesic ray from $o$ representing each class.

\end{Def}

The {horofunction compactification} of $X$ is obtained by attaching the visual boundary $\geo X$. More precisely, for $x\in X$, define $d_x: X \to \R$ be given by $$ d_x(y) = d(y,x) - d(o,x).$$
If $x_n$ is a sequence in $X$ converging to $\xi \in\geo X$ (with $\xi$ represented by a ray from $o$),
then
\[
 d_{x_n} \to b_\xi
\]
uniformly on compact sets of $X$ \cite[Chapter II.8]{BH}.

The next lemma reduces the proof of Theorem \ref{hd} to controlling the Busemann functions $b_\xi|_Y$ at every accumulation point $\xi\in \partial_\infty X$ of the sequence $\ga_t(s)$.

    \begin{lemma}\label{lem:proj_converges}
        Let $\xi\in \partial_\infty X$ and $x_n\in X$ be a sequence converging to $\xi$ in the visual topology. Suppose  $Z\subset X$ is a closed convex set and $b_\xi\vert_Z$ has a unique minimum at $z_0\in Z$. Then the nearest point projection map $\operatorname{pr}: X \to Z$ satisfies $\operatorname{pr}(x_n) \to z_0$ as $n\to \infty$.
    \end{lemma}

    \begin{proof}
        Without loss of generality, assume that $b_\xi\vert_Z(z_0) = 0$.
        For $r>0$, let  $B_r(z_0)$  be the closed ball of radius $r$. 
        Let $$\delta\coloneqq \min\{b_\xi(z):\ z\in \partial B_r(z_0) \}>0.$$ Since $d_{x_n} \to b_\gamma$ uniformly on compacts, for all large enough $n$, we have
        \[
         d_{x_n} (z_0) 
         \le \delta/3, \quad 
         d_{x_n} (z)  \ge 2\delta/3  \text{ for all } z\in \partial B_r(z).
        \]
        Since $d_{x_n}$ is strictly convex along geodesics in $X$, for all $n\ge n_0$, $d_{x_n}$ must achieve its unique minima in $B_r(z_0)$. Taking $r\to 0$, we finish the proof.
    \end{proof}

\subsection*{Properness via relative position in $G/P$}
Any unit-speed geodesic ray in $X$ has the form $t\mapsto  g\exp (tv)o $ for some $g\in G$ and a unit vector $v\in \fa^+$.
The ray is called regular if $v\in \inte \fa^+$ and singular otherwise. A point $\xi\in \partial_\infty X$ is called regular and asymptotic to $f\in G/P$ if it is represented by a regular geodesic ray $g\exp (tv)o$, $t\in [0, \infty)$ with $f=gP$. 
Denote the set of all regular points in $\partial_\infty X$ by $\partial_\infty^{\rm reg} X$. There is
 a well-defined map
\begin{equation}\label{eqn:f}
    f:\partial_\infty^{\rm reg} X \to G/P,
\end{equation}
sending $\xi$ to $f_\xi \coloneqq gP$.

        The Weyl group $W = \{e,w_1,w_1',w_2,w_2',w_0 \}$ of $G$ is depicted in \Cref{fig:weyl_gp}.
        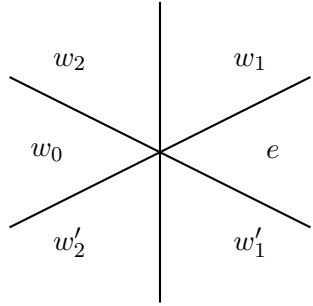
\begin{figure}[ht]
        \centering
        \begin{tikzpicture}
            \draw[thick] (-2,-1) -- (2,1);
            \draw[thick] (-2,1) -- (2,-1);
            \draw[thick] (0,-2) -- (0,2);
        
            \node at (-1.2, 1.2) {$w_2$};
            \node at (1.2, 1.2) {$w_1$};
            \node at (1.5, 0) {$e$};
            \node at (1.2, -1.2) {$w_1'$};
            \node at (-1.2, -1.2) {$w_2'$};
            \node at (-1.5, 0) {$w_0$};
        \end{tikzpicture}
        \caption{The Weyl group.}
        \label{fig:weyl_gp}
        \end{figure}
        Here $w_0$ denotes the longest Weyl element
        so that $w_0 P w_0^{-1}$ is the lower triangular subgroup.
       The Schubert cell decomposition of $G/P$:
        \[
         G/P = \bigsqcup_{w\in W} P w P,
        \]
       with the unique open $P$-orbit  $P w_0 P$ (a 3-cell); two 2-cells $P w_2 P$ and $P w'_2 P$; two 1-cells $P w_1 P$ and $P w'_1 P$; and the zero cell is $P$. Identifying $G/P$ with the full flag variety $\{(p, \ell): p\in \R P^2, \ell \subset \R P^2 \text{ a line }, p\in\ell\}$, the Schubert cells relative to a given  $(p,\ell)\in G/P$ can be characterized as follows:
        \begin{enumerate}
            \item  0-cell: $\{(p,\ell)\}$.
            \item 1-cells:  $\{((p',\ell') :\ p=p' \text{ and } \ell\ne \ell' \}$, $\{((p',\ell') :\ p\ne p' \text{ and }\ell=\ell' \}$.
            \item 2-cells: $\{(p',\ell') :\ p\in \ell'- p'\}$, $\{(p',\ell') :\ p'\in \ell - p\}$.
            \item 3-cell: $\{(p',
        \ell') :\ p\not\in \ell' \text{ and } p'\not\in \ell\}$.
        \end{enumerate}
\begin{Def}\label{rel} For $1\le k\le 3$,  two points $\xi=(p, \ell), \xi'=(p',\ell') \in G/P$ are in {\em relative position $k$} if $(p',\ell')$ lies in a $k$-cell in the Schubert cell decomposition of $G/P$ with respect to $(p,\ell)$, or equivalently, if there exists $g\in G$ such that $\xi=gP$ and $\xi=gwP$ for $w\in\{w_k, w_k'\}$ where we have put  $w_3=w_3'=w_0$.
\end{Def}

Since any geodesic ray in $Y$ is regular, we have
$\partial_\infty  Y\subset \partial_\infty^{\rm reg} X$, and $f(\partial_\infty Y)=\Lambda_Y= K_0 P/P$.

    \begin{lem}\label{lem:busemann_proper}
        Let $\xi\in  \partial_\infty^{\rm reg} X$.
        If $f_\xi$ has relative position 2 or 3 with respect to every point of $\Lambda_Y$, then $b_\xi|_Y$ is proper and bounded below, and 
        it attains a unique minimum in $Y$.
    \end{lem}
\begin{proof}

 Busemann functions are convex, so to prove  properness and boundedness below, it suffices to show that $b_\xi\to \infty $ along every ray in $Y$ issuing from $o$. Let $\{\xi_t:t\ge 0\}$ be a ray from $o$ representing $\xi\in \partial_\infty X$, and
      let $\{\rho_t:t\ge 0\}$ be a ray in $Y$ from $o$. Let $\rho\in \partial_\infty Y$ represented by it. By the hypothesis,
    there exists $g\in G$ such that we have $f_\rho=g P$
    and $f_\xi= gwP$ for some $w\in \{w_2, w_2', w_0\}$.
     Consider the maximal flat $F:=gAo$ and its two Weyl chambers
     $W_-:=g w A^+o$ and  $W_+:=gA^+o$ asymptotic to $f_\xi$ and $f_\rho$, respectively. Set $x=go$. 
     Let $\xi'$ (resp. $\rho'$) be the central ray in $W_-$ (resp. $W_+$) from $x$, asymptotic to $\xi$ (resp. $\rho$). 
       Since $\xi'$ and $\xi$ are asymptotic, their Busemann functions $b_{\xi'}$ and $b_{\xi}$ differ only by an additive constant.
        Moreover, since Busemann functions are 1-Lipschitz, showing that $b_{\xi}$ goes to infinity along $\rho$ is equivalent to showing that $b_{\xi'}$ goes to infinity along $\rho'$.
        
       Note that the restriction $b_{\xi'}|_F$ is the Busemann function on the Euclidean plane $F$ corresponding to the ray $\xi'$,
        which must go to infinity along any ray which makes an angle strictly more that $\pi/2$ with $\xi'$. Since $w\in \{w_2, w_2', w_0\}$,  $W_-$ and $W_+$ are not adjacent and hence the angle between $\rho'$ and $\xi'$ is strictly bigger than $\pi/2$.
    Therefore $b_{\xi}$ goes to infinity along $\rho$.

     The second part of theclaim that $b_\xi|_Y$ has a unique minimum   follows from the first part.  Indeed, by the first part, $b_\xi|_Y$ has a minimum. If $b_\xi|_Y$ had two distinct minima $y_1, y_2 \in Y$, then convexity would imply that $b_\xi$ is constant along the geodesic segment in $X$ joining $y_1$ and $y_2$. Since Busemann functions are real analytic, it would then follow that $b_\xi|_Y$ is constant along the complete bi-infinite  geodesic extension of that segment, contradicting properness.
    \end{proof}

\subsection*{Uniform regularity and properness}

\begin{lem}\label{ur}
If $t_n\to\infty$, $\ga_{t_n}(s_n) = a_{t_n} k_0 h_{s_n}o \to \infty $  uniformly regularly as in Def. \eqref{dur}, then 
$$\pi(\ga_{t_n}(s_n))\to o\quad\text{as $n\to\infty$.}$$
    
\end{lem}
\begin{proof}
Since  $\ga_{t_n{(s_n)}}\to \infty$ regularly, $|s_n|\to\infty$ by Corollary \ref{cor:regular}.
After passing to a subsequence, we may assume that $\ga_{t_n{(s_n)}}$ converges to some $ \zeta\in G/P$ in the sense of Definition \ref{dur}.   By \Cref{every}, $\zeta=(p_0,\ell_0)$, where $p_0 =[ e_1+e_3]$ and $\zeta \in \Lambda_{\pi^{-1} (o)}.$

The quadric $y^2=2xz$ splits $\br  {P}^2$ into the disk 
$D = \{[ x:y: z] :  y^2<2xz \}$ and the M\"obius strip $\{[ x:y: z] :  y^2> 2xz \}$. 
Since by \Cref{lem:Lambda_Y}, $$\Lambda_Y =\{ (p,\ell) : p\in \partial D, \ell \text{ is tangent to } \partial D \text{ at } p\}$$ and $p\in D$,
it follows that $\zeta$ has relative position $2$ or $3$ with respect to any point in $\Lambda_Y$. Since  $\ga_{t_n}(s_n) \to \infty $ uniformly regularly, the  accumulation set of
$\ga_{t_n}{(s_n)}$ is a compact subset $C \subset \geo^{\rm reg} X$ with $f(C) = \zeta$ (see \eqref{eqn:f}).
By  \Cref{lem:busemann_proper}, for each $\xi\in C$, $b_\xi\vert_{Y}$ has a unique minimum in $Y$. 

We claim that this minimum is achieved at $o$, which would finish the proof by \Cref{lem:proj_converges}.
By Corollary \ref{fiber2},
 $$\pi^{-1}(o)=\bigcup_{k\in K_0, w\in \cal W} kk_1 w A^+o$$
is a union of Weyl chambers emanating from $o$.
Let $\Delta$ be the Weyl chamber emanating from $o$ and asymptotic to $\zeta$. 
Since $\zeta \in \Lambda_{\pi^{-1} (o)}$, we have $\Delta\subset \pi^{-1}(o)$. 
After extraction, $\g_{t_n}(s_n)$ is asymptotic to a ray $\gamma$ in the Weyl chamber $\Delta$ in $\pi^{-1}(o)$ emanating from $o$ (i.e., $\g_{t_n}(s_n)\in X$ converges to $[\xi]\in \geo X$ in the compactification $X\sqcup \geo X$).
This ray $\gamma$ must be perpendicular to $Y$ at $o$. Thus the Busemann function $b_\gamma$ attains its minimum at $o$. 
\end{proof}

\begin{lem}\label{nour}
If  $t_n\to \infty$ and $\ga_{t_n}(s_n)$ has no subsequence which tends to $\infty$ uniformly regularly, then $$d(L,\pi(\ga_{t_n}(s_n)))\to 0 \quad\text{ as $n\to\infty$}.$$
\end{lem}
\begin{proof} 
By \Cref{cor:regular}, we have $|s_n|/t_n\to 0$ as $n\to \infty$.
In this case, $\g_{t_n}(s_n)$ converges to the ray $\xi:=\{a_{t}o:t\ge 0\}$ in the visual topology: to see this, we consider the right-angled triangle $\triangle (o,a_{t_n}o, \g_{t_n}(s_n))$. Since $X$ is a CAT(0)-space, this triangle is thinner than a euclidean triangle with the same side lengths,
\[
 \angle_{o}(a_{t_n}o, \g_{t_n}(s_n)) \le \tan^{-1}\frac{|s_n|}{t_n} \to 0\quad \text{as }n\to\infty.
\]

Therefore,  $\g_{t_n}(s_n)$ converges to the ray $\xi:=\{a_{t}o:t\ge 0\}$ in $\partial_\infty X$ and $d_{\ga_{t_n}(s_n)}\to b_\xi$.

Since $\pi(a_{t_n}o)=o$ and $\pi$ is $1$-Lipschitz,
we have $$d(\pi(\g_{t_n}(s_n)), o) \le d(a_{t_n} k_0 h_{s_n} o, a_{t_n}o) = |s_n|.$$

Since $Y=(H\cap A) k_0 L$, 
there exists $s_n'\in \br$ such that 
\be\label{sp} h_{s_n'} \pi(\ga_{t_n}(s_n)) \in  k_0 L .\ee 
Since $h_{-s_n'}k_0L$ is orthogonal to $L$ at $h_{-s_n'}o$,
$h_{-s_n'}o$ is the nearest projection of
 $\pi(\ga_{t_n}(s_n)) $ to $L$, and
 hence
 $$d(\pi(\ga_{t_n}(s_n)) , h_{-s_n'}(o))\le  d(\pi(\ga_{t_n}(s_n)) , o)\le |s_n| .$$
Hence by the triangle inequality,
$$|s_n'|=d(o, h_{-s_n'}(o))\le  d (o, \pi(\ga_{t_n}(s_n)))+ d(\pi(\ga_{t_n}(s_n)) , h_{-s_n'}(o))\le 2|s_n|.   $$
(In fact, since the geodesic triangle in $Y$ with vertices $h_{-s_n'}(o)$,
 $ \pi(\ga_{t_n}(s_n)), o$ has the right angle at $h_{-s_n'}(o)$, we even get
 $|s_n'|\le |s_n|$).

 Since $2|s_n|/t_n \to 0$ as $n\to\infty$, we conclude again that 
\[
 d_{h_{s_n'}\g_{t_n}(s_n)} \to b_{\xi},
\]
uniformly on compacts.

 Since the ray $\xi = \{a_t o:\ t\ge 0\}$ is perpendicular to $k_0L$ at $o$, the point $o$ is a minimum of $b_{\xi}$ on $k_0 L$.  
We claim that this minimum is unique:  
Suppose that  $y\in k_0 L$ is another point such that $b_\xi(o) = b_\xi (y)$.  
Let $\{\rho_t :\ t\ge 0\}$ be the ray in $X$ emanating from $y$ and asymptotic to $\xi$ and consider the Busemann function 
$$\tilde b_\rho: X\to\R:\; \tilde b_\rho (x)= \lim_{t\to \infty} d(\rho_t, x)- t.$$
Then  $\tilde b_\rho -b_\xi$ is a constant function.  Since $y$ is a minimum of $ b_\xi$ on $k_0 L$, it must be a minimum of $\tilde b_\rho$.
Thus $\rho$ must be orthogonal to $k_0 L$ at $y$.
Since $\xi$ and $\rho$ are perpendicular to the segment $[o,y]$ at its endpoints and are asymptotic, it follows from the  Flat Strip Theorem (see \cite[Ch. I, Cor. 5.8(i)]{Ballmann}) that $\xi$ and $\rho$ bounds a flat half strip; in particular, $\xi$ and $\rho$ (and hence the segment $oy$) must lie in a $2$-flat $F\subset X$.  
Thus $F$ contains the ray $\xi = \{a_t o:\ t\ge 0\}$ as well as the geodesic $k_0L = \{k_0 h_t o: t\in\R\}$ (since it contains the segment $oy\subset k_0L$).  
Hence for all $t_0\in\R$, $k_0 h_{t_0}k_0^{-1}$ lies in the centralizer of $\{a_t:\ t\in\R\}$, which can be verified to be false by a straightforward computation. So, we arrive at a contradiction.

Since $o$ is the unique minimum of $b_\xi\vert_{k_0L}$,  \Cref{lem:proj_converges} implies that the nearest projection of $h_{s_n'}\g_{t_n}(s_n)$ to $k_0L$ converges to $o$. By \eqref{sp},
we have
\[
d(\pi(h_{s_n'}\g_{t_n}(s_n)), o) \to 0 \quad \text{as }n\to\infty.
\] 
Since $\pi$ is $H$-invariant, we finally have
\[
 d (\pi(\gamma_{t_n}(s_n)), L) \to 0 \quad \text{as } n\to\infty
\]
as asserted.
\end{proof}

\subsection*{Projection is bounded}
\begin{theorem} \label{bp} For any $\e >0$,  the union $\bigcup_{|t|>\e } \pi(Y_t)$ lies within a bounded distance from $L$.
\end{theorem}

Since $\pi(Y_t) = A_0 \pi(a_t k_0 A_0 o)$, the result follows from the following result:

\begin{lemma}\label{claim} For $\e>0$,
 the set  $\{ \pi(\g_{t}(s)): |t|>\e, s\in \br\}$ lies within a bounded distance from $L$.
 Moreover, if either $t_n$ or $|s_n|$ is bounded, then
 $\pi(\ga_{t_n}(s_n))$ is bounded.
\end{lemma}

For the proof of this lemma, we assume that the parameter $t$ is positive. More precisely, we show:

\begin{lemma}\label{claim2}
For all $\e>0$,
the set $\{ \pi(\g_{t}(s)) : t>\e,\, s\in \br \}$ lies within a bounded distance from $L$.
Moreover, if either $t_n$ or $|s_n|$ is bounded, then
$\pi(\ga_{t_n}(s_n))$ is bounded.
\end{lemma}

This suffices, as can be seen as follows: the Cartan involution $\Theta : G\to G$ given by \eqref{cartan-inv} induces an isometry $\iota: X \to X$ defined by $\iota([gK]) = [(g^T)^{-1}K]$ (a {\em point reflection} about the basepoint $o$). This isometry preserves all totally geodesic subspaces passing through $o$, such as $Y$, $L$, and $Ao$. Note that $\Theta (a_t) = a_{-t}$. Since the nearest point projection $\pi: X\to Y$, viewed as a map from $X\to X$, commutes with $\iota$, \Cref{claim2} also implies that the set $\{ \pi(\g_{t}(s)) : -t < -\e,\, s\in \br \}$ lies within a bounded distance from $L$, thereby concluding the proof of \Cref{claim}.

\begin{proof}[Proof of \Cref{claim2}]
Suppose that the claim is not true. Then there exist $\epsilon>0$ and sequences $s_n$ and $t_n>\epsilon$ such that the sequence $\pi(\gamma_{t_n}(s_n))$ diverges away from $L$.
After extraction, we have the following cases.

\smallskip
\noindent{{\bf Case 0:} Both $s_n$ and $t_n$ are bounded.} Then  $\pi(\g_{t_n}(s_n))$ is bounded, leading to a contradiction.

\smallskip
\noindent{{\bf Case 1:}  $t_n$ is bounded and $|s_n|\to\infty$.} 
Passing to a subsequence, we can assume that $t_n\to t$ and $s_n\to\infty$.

Let $\xi\in\geo X$ denote the equivalence class of the ray $\{\g_t(s):\ s>0\}$. Then, $\gamma_{t_n}(s_n)\to \xi$ as $n\to\infty$.
Thus, we have 
    \[
     d_{\g_{t_n}(s_n)} \to b_{\xi}
    \]
    uniformly on compacts. 
    
We observe that $f_\xi$ (cf. \eqref{eqn:f}) has position either $2$ or $3$ relative to any point in $\Lambda_Y$: To see this, note that if we write $f_\xi = (p,\ell)$ (as described in the paragraph before \Cref{rel}), then $$p = a_t k_0 [e_1] = [e^t e_1 -\sqrt 2 e^{-2t} e_2 + e^t e_3].$$ Thus, for $t>0$, $p$ lies in the interior of the disk $D\subset \R \op{P}^2$  bounded by the conic $y^2 = 2xz$, which implies the assertion.

    Therefore, by
    \Cref{lem:busemann_proper},
      $b_{\xi}\vert _{Y}$ is proper and bounded below,
    and has a unique minimum in $Y$. By Lemma \ref{lem:proj_converges}, it follows that $(\pi(\g_{t_n}(s_n)))$ converges to this unique minimum and thus we again have a contradiction.

\smallskip

\noindent
{\bf Case 2:} If the sequence $s_n$ is bounded, then the sequence $\g_{t_n}(s_n)$ lies in a bounded neighborhood of the singular ray $\{a_t o : t>0\}$, and thus must have a bounded projection to $Y$. This is a contradiction.

\smallskip
\noindent{\bf Case 3:} Finally, if $t_n\to \infty$ and $|s_n|\to \infty$ as $n\to\infty$, then it follows from Lemma \ref{ur} and Lemma \ref{nour} that $\pi(\gamma_{t_n}(s_n))$ is bounded. Again, this is a contradiction.

\medskip
The ``moreover'' part follows from Cases 1 and 2 discussed above.
\end{proof}

\begin{figure} 
  \includegraphics [height=4cm]{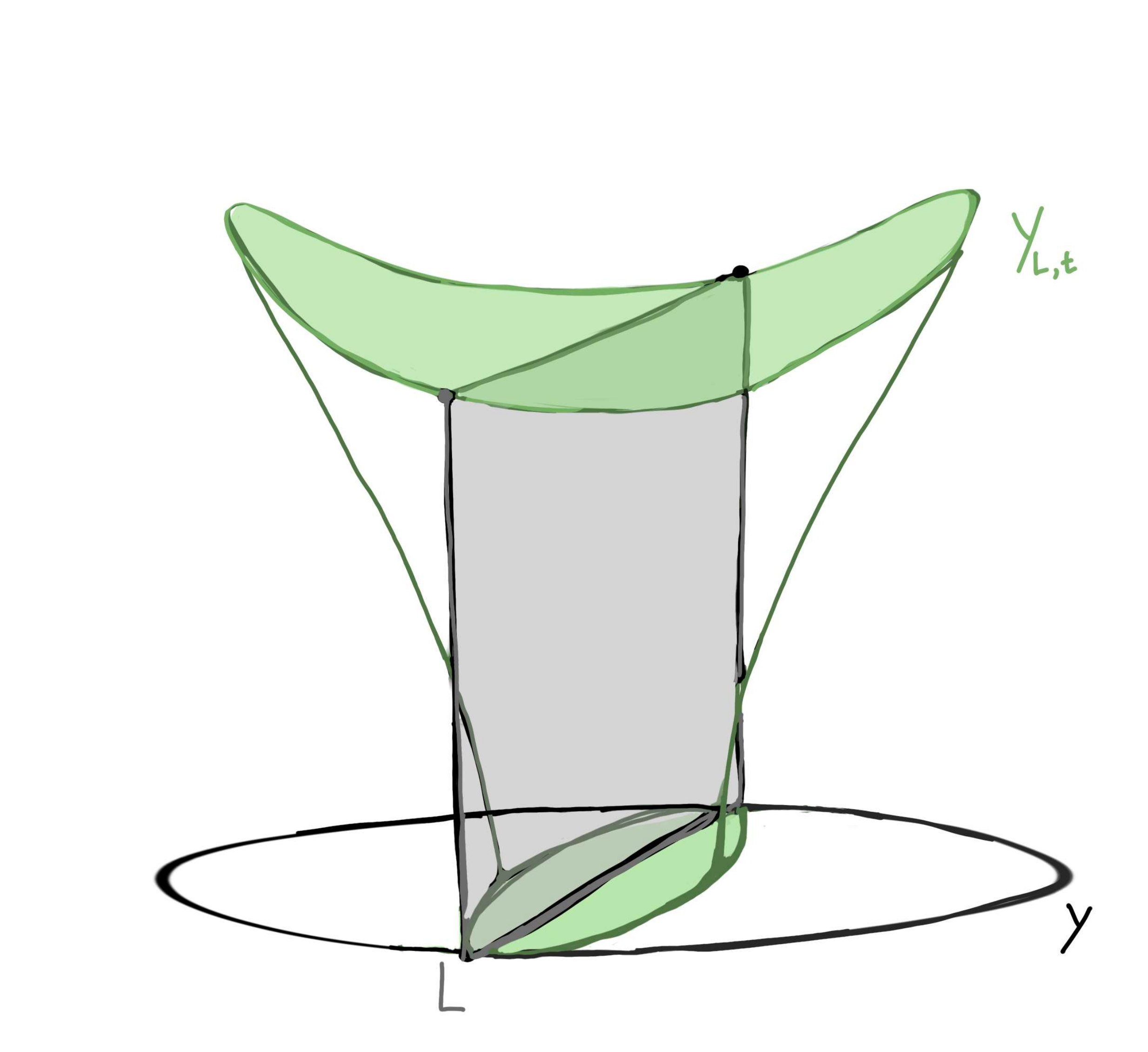}  
\caption{shadow of the floating plane}
\end{figure}

\subsection*{Projection is narrow}
\begin{theorem}\label{narrow}
The Hausdorff distance between
$\pi(Y_t)$ and $L$ goes to $0$ as $|t|\to\infty$.
\end{theorem}

\begin{proof}
Again, we make the assumption that $t > 0$, as the case when $t < 0$ is analogous.

Suppose the claim is not true. Since $L \subset \pi(Y_t)$, there exist $\delta > 0$, $t_n \to \infty$, and $x_n \in Y_{t_n}$ with $d(L, \pi(x_n)) \geq \delta$. As in the proof of Theorem \ref{bp}, write $x_n = h_{s_n'} a_{t_n} k_0 h_{s_n} o$ and use the $A_0$-equivariance to take $s_n' = 0$. Lemmas \ref{ur} and \ref{nour} then give a contradiction.
\end{proof}

\section{Closures of floating planes: The Fuchsian case}
In this section, we consider the closure of an $H$-orbit $[h]a_t H$ in $\Ga \ba G$
where $h\in H$ and $\Ga$ is a cocompact lattice in $H$. 
Let $N^-$ and $N^+$ be the strictly upper and lower triangular subgroups of $G$ and set 
$N_0^{\pm}=H\cap N^{\pm}$.

For any $h\in H$, the product map $hN_0^{+}\times N_0^{-}\times A_0\to H$ is a diffeomorphism onto its image which is Zariski open and dense. Any right $A_0$-invariant open neighborhood of $h\in H$ contains a subset of the form $\cal O=h U^+ U^- A_0$ for some open neighborhood $U^{\pm}\subset N_0^{\pm}$. We will call a subset of this type a basic open subset and
let $\pi_{\pm}=\pi_{\cal O, \pm}: \cal O
\to  N_0^{\pm}$ be the Bruhat projections 
\be\label{pipi} \pi_+(h n^+n^-a)= n^+\quad\text{ and }\quad \pi_-(hn^+n^-a)= n^- .\ee 

\begin{Def}\label{adm}
We say that an $A_0$-invariant subset $Z\subset H$ is admissible if the box dimension of $\pi_{\pm} (\overline{Z}\cap \cal O) $ is equal to its Hausdorff dimension for all sufficiently small basic open subsets $\cal O$. If $L=hA_0 o$ is a geodesic in $Y$, we say $L$ or $\ell=\Ga\ba L$ is admissible if ${\Gamma hA_0}$ is admissible.
\end{Def}
The following theorem says that when the floating height $t$ is non-zero, it can be as chaotic as the closure of its reference geodesic $\Ga\ba \Ga h A_0$ in $\Ga\ba H$.

\begin{theorem}\label{dim}\label{hd2}
    Let $\Ga<H$ be a discrete subgroup and $h\in H$. Let $t\ne 0$.

\begin{enumerate}
    \item We have
$$\op{dim} \overline{\Ga h a_t H }= 2+\op{dim} \overline{\Ga h A_0} .$$
\item Let $Y_{L, t}= ha_tHo$ for $L=hA_0o$. Suppose that $L$
is admissible. Then 
$$ \frac{1}{2}\left(  {3+  \dim \overline{\Gamma hA_0} }\right)  \le  \op{dim} \overline{\Ga Y_{L, t} }\le  1+ \op{dim} \overline{\Gamma h A_0}  .$$
\item In particular, if $1< \dim \overline{\Ga\ba \Ga h A_0}<2 $, then
$$  \op{dim} \overline{\Ga Y_{L, t} } \in (2,3).$$
\end{enumerate}
\end{theorem}

The rest of this section is devoted to proving Theorem \ref{dim}.
We begin by showing that the closure of $[h]a_t H$ in $\Ga\ba G$ is governed by the closure of
$[h]A_0\subset \Ga\ba H$. Note that for $\Ga<H$, the quotient $\Ga\ba H$ is a closed subset of $\Ga\ba G$.
\begin{prop}\label{fuchsian} Let $\Ga<H$ be a discrete subgroup and $h\in H$. For any $t\in \br$, 
  we have
   $$\overline{\Ga h a_t H}   =\overline{ \Ga h A_0} a_tk_0A_0K_0 .$$
   In particular, for $L=hA_0o$,
   $$\overline{\Ga Y_{L, t}}   =\overline{ \Ga h A_0} a_tk_0A_0 o .$$
\end{prop}
\begin{proof}  Since $H=A_0 k_0
A_0K_0$, we have
$${\Ga h  a_t H}  = {\Ga h A_0 a_tk_0A_0K_0}.$$
Let $g\in \overline{\Ga h a_t H} $, i.e., 
$\ga_i h a_t h_i\to g$ for an infinite sequence $\ga_i\in \Ga, h_i\in H$. 
We write $h_i= c_i k_0 p_i k_i\in A_0 k_0A_0 K_0$, so $a_t h_i= c_i a_t k_0 p_i k_i$. 
It suffices to show that
$a_tk_0p_i$ is bounded, so that its limit lies in $a_t k_0A_0K_0$.
Write $$a_tk_0p_i = h_i'b_i k_i'\in HBK$$ using \Cref{affine}.
As $t$ is fixed, $$\pi (a_t k_0 p_io)=h_i' o$$ is bounded by \Cref{claim} and hence
 $h_i'\in H$ is bounded.
 Using \Cref{affine}, we can write $g=h^*b^*k^*\in H B K$.
Since $$\ga_i h a_th_i= \ga_i h c_i (a_t k_0 p_i)k_i= (\ga_i h c_i h_i') b_i (k_i'k_i) \in HBK $$ and the $B$-component in the $G=HBK$ decomposition is uniquely determined modulo a compact subset, we must have
$b_i\to b^*$ modulo a compact subset, and hence $b_i$ must be a bounded sequence.
Since both $h_i'$ and $b_i$ are bounded sequences and
$a_tk_0p_i = h_i'b_i k_i'\in HBK$, it follows that $a_tk_0p_i$ is bounded, as desired.
\end{proof}

To relate the Hausdorff dimension of $\overline{\Ga h A_0}$ with that of $\overline{\Ga h a_t H}$, we use the local product structure in $G$.
The following lemma shows that $H\times a_t k_0A_0K_0$ maps locally
diffeomorphically into $G$, allowing us to
invoke Proposition \ref{fuchsian} in the proof of Theorem \ref{dim}(1).
\begin{lem}\label{gd}
    Let $t\ne 0$.  The product map
    $$m: H\times a_tk_0  A_0K_0 \to G,\quad (h, s)\mapsto hs$$ 
     with $h\in  H$ and $s\in a_tk_0A_0K_0$
   is a local diffeomorphism onto its image at every point, i.e., for any
   $p=(h,s)$,  there exists an open neighborhood $U$ of $p$ such that
   $m(U)$ is a submanifold of $G$ and $m|_U:U\to m(U)$ is a diffeomorphism.
   \end{lem}
\begin{proof}   Write  $g=a_tk_0$ and set 
 $S_t= gA_0K_0 g^{-1}$. To prove the claim, it suffices to show
 that for the product map
    $m: H\times S_t\to G$ given by $(h, s)\mapsto hs$, $dm_{(h,s)}$ is
injective at every $(h,s)\in H\times S_t$.
Then by the constant rank theorem, the claim would follow.

    So, let $h\in H$ and $s=g (ak)g^{-1}\in S_t$ where $a\in  A_0$ and $k\in K_0$.
  For any $U \in \T_h H$ and $V\in \T_s S_t$, we get
  $$(L_{(hs)^{-1}})_* dm_{(h,s)} (U, V)= \Ad_{s^{-1}} (L_{h^{-1}})_* U + (L_{s^{-1}})_* V $$
 where $L_x:G\to G$ denotes the left translation by $x\in G$ and $(L_x)_*: \T_y G\to \T_{xy}G $ denotes the differential at $y\in G$.
Therefore
  $$\op{ker} dm_{(h,s)} \simeq \Ad_{s^{-1}} \frak h \cap (L_{s^{-1}})_* \T_s(S_t).$$
Now $$(L_{(ak)^{-1}})_* \T_{ak}(A_0K_0) =\Ad_{k^{-1}}\frak a_0 \oplus\frak k_0 $$
and conjugating by $g$ gives
 $$(L_{s^{-1}})_* \T_{s}S_t =\Ad_g \left( \Ad_{k^{-1}}\frak a_0 \oplus\frak k_0 \right).$$
Setting $\frak h'=\Ad_{g^{-1}}\frak h$, we can write  $$\Ad_{s^{-1}}\frak h=\Ad_g (\Ad_{a^{-1}k^{-1}} \frak h').$$ 

It remains to show that  
 $$ \Ad_{a^{-1}k^{-1}} \frak h'\cap   \left( \Ad_{k^{-1}}\frak a_0 \oplus\frak k_0\right)=\{0\}.$$

Since $k=k(\theta)$ as in \eqref{kt}, we can compute that any matrix in $\Ad_{k^{-1}}\frak a_0 \oplus\frak k_0$ is of the form
\be\label{int} \begin{pmatrix}
    b & p & 0 \\ q & 0 & p \\ 0 & q & -b
\end{pmatrix}\ee 
for some $b, p, q\in \br$. Moreover, if $k=e$,  then $p=q$.

Now, an element in $\frak h'$ is of the form 
\be \label{hp} g^{-1} \begin{pmatrix}
    s & x & 0 \\ y & 0 & x \\ 0 & y & -s
\end{pmatrix}g= \frac{1}{2}\begin{pmatrix}
    X & Y & Z \\ W & 0 & Y \\ Z & W & -X
\end{pmatrix}\ee 
for some $x,y, s\in \br$, where \begin{itemize}
    \item  $X= -\sqrt 2 (x+y)(e^{3t}+e^{-3t})$, 
\item $Y=2\sqrt 2 s +2e^{3t}(x-y) $, 
\item $Z= \sqrt 2 e^{-3t}(x+y)$,
\item 
$W= 2\sqrt 2 s- 2 e^{-3t}(x-y)$.
\end{itemize}

 Let $k=k_\theta$, $\mathsf c=\cos \theta$, $\mathsf s=\sin \theta$, and $a=h_r$. 
 Then a matrix  $I \in \Ad_{a^{-1}k^{-1}}\frak h'$ is, up to a uniform constant multiple, of the form 
 $$ I_{11}=2\mc  X +\sqrt 2 \ms Y+\sqrt 2 \ms W, \;\;  I_{22}=-\sqrt 2 \mc \ms W, \; \; I_{33}=2\mc Z+\sqrt \ms (1-\mc)W$$
 $$ I_{12}=(-\sqrt 2 \ms X +(\mc +1) Y +\sqrt 2\ms Z - \ms^2 W)e^{-3r},$$
 $$I_{23}=  (-\sqrt 2 \ms X +(\mc +1) Y-\sqrt 2 \ms Z +\mc (1-\mc)W )e^{3r},
  $$
$$ I_{21}= (-\sqrt 2 \ms X +(\mc -1) Y -\sqrt 2\ms Z +(\mc +1) W)e^{3r},$$
$$ I_{32}=  ( -\sqrt 2 \ms X+(\mc -1) Y +\sqrt 2 \ms Z)e^{-3r} , $$ $$I_{13}= 2 \mc Z +\sqrt 2^{-1} (1-\mc) \ms W , \;\; I_{31}= -2\mc Z.$$

Suppose that $I\in \Ad_{a^{-1}k^{-1}} \frak h'\cap   \left( \Ad_{k^{-1}}\frak a_0 \oplus\frak k_0\right)$. 
To show that $I=0$, we consider the following three cases separately.

\medskip
\noindent{\bf Case I: $\mc \, \ms \ne 0$.}
Since $I_{13}=I_{22}=I_{31}=0$, we must have  $Z=W=0=I_{33}$. 
Since $Z=0$,  $x=-y$. Together with $W=0$, this gives $\sqrt 2 s= 2 e^{-3t} x$.
Since $I_{33}=0$, $I_{11}=0$, i.e, $\sqrt 2 \mc X+\ms Y=0$, and this means that $\sqrt 2 s=-2 e^{3t}x$.
Since $t\ne 0$, we get $x=0$, which also means that $s=y=0$. Therefore $I=0$.

\medskip
\noindent{\bf Case II: $\mc=0$.} Then from $I_{13}=0$, we get $W=0$, which implies $I_{33}=0$.
From $I_{11}=-I_{33}$, we get $Y=0$. This implies that $x=y=s=0$. Hence $I=0$.

\medskip
\noindent{\bf Case III: $\ms=0$.} In this case, $k=e$. So $I_{12}=I_{21}$.
Then from $I_{13}=0$, we get $x+y=0$. Since $I_{32}=0$, we get $I_{21}=0$, which gives us $W=0$, and hence $Y=0$.
This implies $x=y=s=0$; so $I=0$.

\medskip
This finishes the proof.
 \end{proof}

\noindent{\bf Proof of Theorem \ref{dim}(1)}
By Lemma \ref{gd}, the product map $f: H \times a_t k_0 A_0K_0 \to G$ is locally
bi-Lipschitz on a countable cover.  Since Hausdorff dimension is countably stable, it follows that
for any subset $\Sigma$ of  $H \times a_t k_0 A_0K_0$, the Hausdorff dimension of $\Sigma$ is equal to that of its image under $f$.

Lemma \ref{fuchsian} gives
     $$\overline{\Ga h a_t H}   =\overline{ \Ga h A_0} a_tk_0A_0K_0 ,$$
and $\overline{ \Ga h A_0}\subset H$ as $\Gamma\subset H$,
Thus the Hausdorff dimension of $\overline{\Ga a_t H} $  is equal to that of the product
$\overline{ \Ga h A_0} \times a_tk_0A_0K_0$.
Since  $ a_tk_0A_0K_0$ is a $2$-dimensional smooth submanifold, the claim follows.

\subsection*{Floating geodesic planes} 
For Part (2) of Theorem \ref{dim}(2), we need an analogue of \Cref{gd} at the level of the symmetric space $X$. Unfortunately, the product map $H\times a_tK_0A_0 o \to X$ is in general not locally injective, and this is precisely why we cannot conclude 
$\dim \overline{\Ga Y_{L, t}}= 1+\dim \overline{\Ga h A_0}$ in Theorem \ref{dim}(2). The next lemma shows that replacing $H$ by $hN_0^{\pm}A_0$ restores local injectivity in $X$, and this will allow us to prove the dimension estimates in Theorem \ref{dim}(2).
\begin{lem}\label{mul2}
    Let $h\in H$. The multiplication map  
    $$ h N_0^{\pm}A_0 \times a_tk_0  A_0o \to X,\quad (h h', so)\mapsto hh's o$$
    with $h'\in  N_0^{\pm}A_0$ and $s\in a_tk_0A_0$
   is a local diffeomorphism onto its image everywhere.
\end{lem}
\begin{proof} Recall the Cartan decomposition $\frak g=\frak k \oplus \frak p$.  Set $H^{\pm}=H\cap N^{\pm}A$ and $\frak h^{\pm}:=\Lie (H^{\pm})$. Let $g=a_tk_0$.
It suffices to show that  the multiplication map
$ \Phi: g^{-1}H^{\pm}g \times  A_0o \to X$ is a local diffeomorphism at $(e, h_ro)$ for any $r\in \br$, since the left-translation by an element of
$g^{-1}H^{\pm} g $ is an isometry.  The image of $d\Phi_{(e, h_r o)}$ is given by
$ dL_{h_r} ( \Ad_{h_r^{-1}g^{-1}}
(\frak h^{\pm})_{\frak p} +\fa_0 )$. Thus it is enough to show
\be\label{int2} (h_r^{-1}g^{-1}\frak h^{\pm} gh_r)_{\frak p} \cap \fa_0=\{0\}\ee 
where $(\cdot )_{\frak p}$ denotes the projection to $\frak p$.

Since the projection $\frak g\to \frak p$ is given by $u\mapsto (u+u^T)/2$,
an element of $(g^{-1}\frak h^{+} g)_{\frak p} $ is of the form as given in \eqref{hp} with 
\be \label{hp2}h_r^{-1} g^{-1} \begin{pmatrix}
    s & 0 & 0 \\ y & 0 & 0 \\ 0 & y & -s
\end{pmatrix}gh_r = \frac{1}{4}\begin{pmatrix}
    2X & Y +W& 2e^{-2r} Z \\ e^r(Y+W) & 0 & e^{-r}(Y+W) \\ 2e^{2r} Z & Y+W & -2X
\end{pmatrix}\ee 
 where \begin{itemize}
    \item  $X= -\sqrt 2 y(e^{3t}+e^{-3t})$, 
\item $Y=2\sqrt 2 s -2e^{3t}y $, 
\item $Z= \sqrt 2 e^{-3t}y$,
\item 
$W= 2\sqrt 2 s+ 2 e^{-3t}y$.
\end{itemize}
Hence $2e^{-2r}Z=0$ and $Y+W=0$ implies that $y=0=s$, proving \eqref{int2} for $\frak h^+$.
The computation for $\frak h^-$ is analogous.
\end{proof}

\begin{lem}\label{larger}
  For $\cal G=\Gamma hA_0$ with $h\in H$, suppose that $\overline{\cal G}$ is admissible. Then
    $$ \max \{ \dim \pi_{+}\left( \overline{\cal G}\cap \cal O \right), \dim \pi_{-}\left( \overline{\cal G} \cap \cal O\right) \}\ge \frac{ \dim (\overline{\cal G})-1} 2$$
    where the supremum is taken over all basic open subsets $\cal O$  of $H$.
\end{lem}

\begin{proof}  Since $H$ can be covered by countably many basic open sets,
$$\dim (\overline{\cal G})=\sup_{\cal O} \dim (\overline{\cal G}\cap \cal O) .$$
Hence if $\dim (\overline{\cal G}) >c_0 $,
then for some basic open subset $\cal O$, we must have  $\dim (\overline{\cal G}\cap \cal O) >c_0$. Write $\cal O=hU_0^+U_o^-A_0$ and
$\overline{\cal G}\cap \cal O= h \bigcup_\alpha  (n_\alpha^+n_\alpha^-) A_0$
with $n_\alpha^{\pm}\in U_0^\pm$. Hence $\pi_{\pm}\left( \overline{\cal G}\cap \cal O \right) =\cup_\alpha n_\alpha^{\pm}$.
   Since $\overline{\cal G}$ is admissible, we have
    $$c_0< \dim (\overline{\cal G}\cap \cal O)\le \dim \left(\bigcup (n_\alpha^+n_\alpha^-) \right)+1
    \le \dim \bigcup n_\alpha^+ +\dim \bigcup n_\alpha^- +1.$$ This inequality forces at least one of  
    $\dim \bigcup n_\alpha^+ $ or $\dim \bigcup n_\alpha^-$ to be bigger than $(c_0-1)/2$, giving the claim.
\end{proof}

\noindent{\bf Proof of Theorem \ref{dim}(2)} 
   By Lemma \ref{fuchsian}, 
   $$\overline{\Ga h a_t H o}   =\overline{ \Ga h A_0} a_tk_0A_0o .$$
   Hence the upper bound is immediate. For the lower bound, by \Cref{larger},
   it suffices to show that for any basic open subset $\cal O\subset H$, we have 
$$\max \{ \dim \pi_{+}\left( \overline{\cal G}\cap \cal O \right), \dim \pi_{-}\left( \overline{\cal G} \cap \cal O\right) \} 
+2 \le  \op{dim} \overline{\Ga Y_{L, t} }. $$
write $\cal O=h_0 U^+ U^- A_0$ for some open neighborhood $U^{\pm}\subset N_0^{\pm}$ and $h_0\in H$. Write $$\overline{\Ga h A_0} \cap \cal O =\bigcup_{\alpha} h_0 n^+_\alpha n^-_\alpha A_0 $$
which is a disjoint union of $A_0$-orbits with $n^{\pm}_\alpha\in U^{\pm}$.
So 
 \begin{multline*}
     \dim \overline{\Ga h a_t Ho }\ge \dim \left(\overline{\Ga h A_0} \cap \cal O\right)a_t k_0A_0 o  \ge 
     \dim \left( (\bigcup h_0 n^+_\alpha A_0 ) a_tk_0 A_0o \right)
 \end{multline*}
 Since $(\bigcup n^+_\alpha A_0 ) \subset N_0^+A_0$, by Lemma \ref{mul2}, we have
 $$\dim \left( h_0(\bigcup n^+_\alpha A_0 ) a_tk_0 A_0o \right)=
 \dim \left( h_0\bigcup n^+_\alpha A_0\right) +1 $$
 which is equal to
 $$ \dim \left(  \bigcup_\alpha n^+_\alpha \right) +2= 
 \dim \pi_+  \left( \overline{\Ga h A_0}\cap \cal O\right)  +2  .$$ 
 Hence $$  \op{dim} \overline{\Ga h a_t Ho }\ge  \dim \pi_{+}\left( \overline{\cal G}\cap \cal O \right)+2 .$$
 The statement $  \op{dim} \overline{\Ga h a_t Ho }\ge  \dim \pi_{-}\left( \overline{\cal G}\cap \cal O \right)+2 $ can be proved similarly.

\section{Bulging deformations and floating planes}
Let $\Gamma<H$ be a torsion-free cocompact lattice and let
$$S=\Gamma\ba Y$$ be the closed orientable hyperbolic surface.
Let $\rho_0 : \Gamma \to H$ denote the inclusion map. Fix a diagonalizable element $\delta \in \Gamma$ representing the homotopy class of an essential simple closed curve $\beta \subset S$. We describe the notion of bulging deformations of $\Gamma$ in $G$, introduced by Goldman \cite{Go-bulging}.

Geometrically, a bulging deformation along $\beta$ alters the convex $\R{P}^2$-structure on $S$ by inserting a projective ``bulge'' along $\beta$. This is achieved by deforming the holonomy representation using a one-parameter subgroup of projective transformations that fix the endpoints of the holonomy of $\beta$ while ``stretching'' transversely to it.

We give a more precise description of the holonomy representation. Suppose first that 
$\beta$ is separating. In this case, the complement of $\beta$ in $S$ consists of two 
connected subsurfaces whose closures we denote by $S_1$ and $S_2$, with $\beta$ as their 
common boundary. The inclusion maps $\beta \hookrightarrow S_i$ ($i=1,2$) induce a 
decomposition of $\Gamma$ as an amalgamated free product
\[
\Gamma = \Delta_1 *_{\langle \delta \rangle} \Delta_2,
\]
where $\Delta_i = \pi_1(S_i)$ for $i=1,2$, and $\langle \delta\rangle$ is the image of $\pi_1(\beta)$ under the inclusion maps, viewed as a common subgroup of $\Delta_1$ and $\Delta_2$.

Let $B$ denote identity component of the centralizer of $\delta$ in $\SL_3(\br)$, which is a maximal real split torus. For any $\mathbf b\in B$, we have a unique homomorphism $\rho_{\mathbf b} : \Gamma \to G$ extending \be  \rho_{\b}(\gamma) = \begin{cases} \gamma & \quad\text{ for $\gamma \in \Delta_1$ } \\
\b \gamma \b^{-1} & \quad\text{ for $\gamma \in \Delta_2$.}\end{cases}\ee 

\begin{figure}[ht] 
  \includegraphics [height=2.5cm]{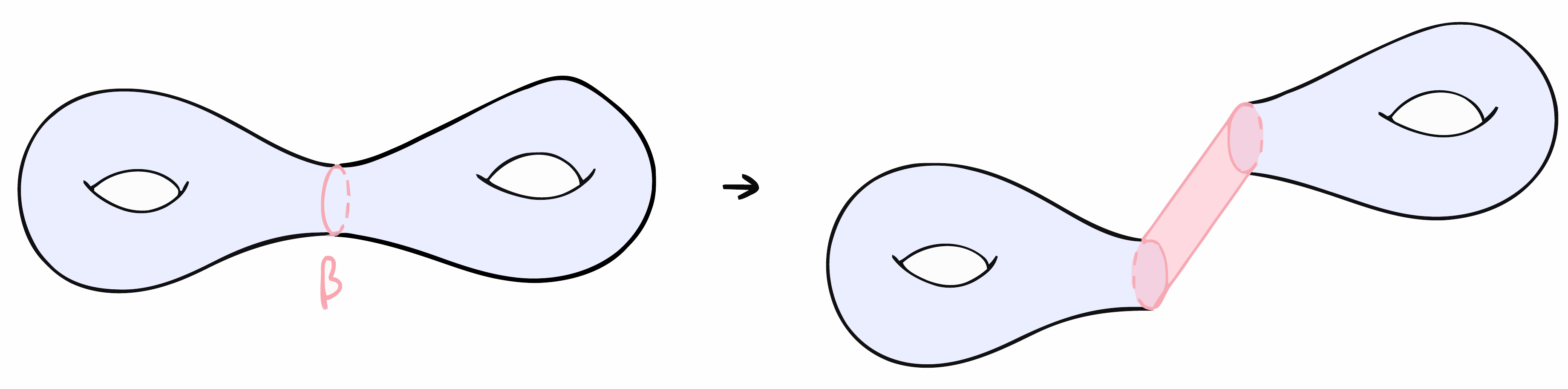}  
\caption{Bulging deformation}
\end{figure}

Now we discuss the case when $\beta \subset S$ is non-separating. Cutting along $\beta$ gives a surface $S_1$ 
with boundary components $\beta_1,\beta_2$, and re-gluing by an orientation-reversing 
homeomorphism $f:\beta_1 \to \beta_2$ recovers $S$. Setting $\Delta = \pi_1(S_1)$, the group 
$\Gamma=\pi_1(S)$ is an HNN extension
\be\label{eqn:HNN}
\Gamma \cong \Delta *_{\psi} \coloneqq 
\langle \Delta, t \mid t b t^{-1} = \psi(b), \;\delta \in \iota_1(\pi_1(\beta_1)) \rangle,
\ee
where $\psi = \iota_2 \circ f_* \circ \iota_1^{-1}$, with 
$f_*: \pi_1(\beta_1) \to \pi_1(\beta_2)$ and 
$\iota_i: \pi_1(\beta_i) \hookrightarrow \Delta$ the induced maps.

The group $\Delta$ naturally embeds in $\Gamma$, so we can view $\Delta$ as a subgroup of $\Gamma$.
Let $\delta$ be a generator for the image of $\pi_1(\beta_1)$ in $\Delta\subset \Gamma$ and let $B=C_G(\delta)$  as above. For any $\mathbf{b}\in B$, we have a unique homomorphism $\rho_{\b} : \Gamma \to G$ extending 
\be
\rho_{\b}(\gamma) = \begin{cases} \gamma & \quad\text{ for $\gamma \in \Delta$ } \\
\b t & \quad\text{ for $\gamma= t$.}\end{cases}
\ee

\medskip
If  $h\in H$ is such that $h\delta h^{-1}\in A_0$, then  $hBh^{-1}=A$, and hence
$$ h \mathsf b h^{-1} = a_{c_0} h_{d_0} \quad\text{ for some $c_0, d_0\in \br$}. $$
 We will call $|c_0|$ the {\em width} of $\b$, which will be denoted by \be\label{bs} \wdt (\b). \ee

We set
\be\label{hit}
\Gamma_{\b} = \rho_{{\mathsf b}}(\Gamma)<G.
\ee

\subsection*{Hitchin property and Zariski density} 
If $\b=\exp u$ for $u\in \frak g$, then for $s\in \br$, set
\begin{equation}\label{eqn:bs}
\b_s=\exp su,
\end{equation}
and consider
the one-parameter family of deformations
$$ \quad  \rho_s:=\rho_{\b_s} .$$

Noting that $h\b_sh^{-1}=a_{c_0s} h_{d_0s} $ for all $s\in \br$,  the width of $\b_s$ is $\wdt(\b_s) = |s|\wdt(\b) = |c_0s|$. 

Clearly, $\{\rho_{s}:s\in \br\}$ lies in the Hitchin  component of $\Hom(\Ga, \SL_3(\br))$, which is the connected component containing $\rho_0$.
Therefore, according to Choi-Goldman \cite{CG}, $\rho_{\b}$ is discrete and faithful. Later, Labourie \cite{La} initiated the theory of Anosov representations and showed each representations in the Hitchin component is Anosov.

\begin{theorem}
    For all $\b\in C_G(\delta)^\circ$, $\Ga_{\b}$ is an Anosov (in particular, discrete) subgroup of $G$. Moreover, if $\wdt({\b}) \ne 0$,
    then $\Ga_\b$ is Zariski dense in $G$.
\end{theorem}

 We justify the ``Zariski dense'' part in the above result only in the case when $\beta\subset S$ is separating, the non-separating case is similar: Note that $\Gamma_\b$ contains $\Delta_1$ and $\b \Delta_2 \b^{-1}$. Since each $\Delta_i$ is Zariski dense in $H$,
 the Zariski closures of $\Ga_\b$ contains both $H$ and $\b H \b^{-1}$. If $\wdt(\b) \ne 0$, then $\b \not\in N_G(H)$. Since $H$ is a maximal connected Lie subgroup of $G$, it follows that $\Ga_\b$ is Zariski dense in $G$.
 We refer to \cite[Sec. 1.2]{DK-amalgam} for a direct proof of the Anosov property of $\Ga_\b$.

\subsection*{Proper embedding away from the bulging locus}
Fix a complete geodesic $\tilde \beta$ in $Y$ which projects to $\beta$.
Every representation $\rho_\b :\Gamma \to G$ admits a $\rho_\b$-equivariant locally isometric map 
\[
 \phi_{\b} : (Y - \Gamma \cdot \tilde \beta) \to X,
\]
constructed as follows. 
Consider the dual graph $T$ (actually a tree) to the lamination $\Gamma \cdot \tilde\beta\subset Y\cong \H^2$, whose each vertex uniquely correspond to a connected component of $Y - \Gamma \cdot \tilde\beta$, and there is an edge between two vertices if the corresponding connected components of $Y - \Gamma \cdot \tilde\beta$ are adjacent. 

If $\beta$ is separating, then there are two $\Gamma$-orbits of the connected components of $Y - \Gamma \cdot \tilde\beta$. So, the vertices of $T$ is bicolored. Note that $T$ is precisely the {\em Bass-Serre tree} associated with the amalgamated  free product decomposition $$\Gamma=\Delta_1 *_{\ga_0} \Delta_2.$$ We fix an edge $e = [v_1,v_2]$ in $T$, i.e. a fundamental domain for the action $\Gamma \curvearrowright T$ and let $Y_1$ and $Y_2$ be the connected component of $Y - \Gamma \cdot \tilde \beta$ corresponding to $v_1$ and $v_2$, respectively. We may assume that $\Delta_1$, $\Delta_2$, and $\langle \g_0\rangle$ are the stabilizers in $\Gamma$ of $v_1$, $v_2$, and $e$, respectively.

If $\beta$ is non-separating, the dual graph $T$ to $\Gamma\cdot \tilde\beta \subset Y$ is again the 
Bass–Serre tree of the HNN extension in \eqref{eqn:HNN}. The connected components of 
$Y-\Gamma\cdot \tilde\beta$ (correspond to vertices of $T$) lie in a single 
$\Gamma$-orbit; we choose the component $Y_1 \subset Y-\Gamma\cdot \tilde\beta$ stabilized by 
$\Delta$. We fix the notation $v_1$ and $v_2$ to denote the (adjacent) vertices in $T$ corresponding to $Y_1$ and $Y_2\coloneqq tY_1$, respectively.

Define $\phi_{\b}: (Y - \Gamma \cdot \tilde \beta) \to X$ by
$$\phi_{\b}|_{Y_1} =i_{Y_1} ,\quad \phi_{\b}\vert _{Y_2} = \mathsf b \circ i_{Y_2}$$
where  $i_{Z} : Z \hookrightarrow X$ denotes the inclusion map.

Using the $H$-equivariant nearest-point projection map $\pi : X \to Y$,
we extend $\phi_{\b}$ to a $\rho_{\b}$-equivariant local isometry
$$F_{\b}: (X - \pi^{-1}(\Gamma \cdot \tilde\beta))\to X$$
by setting $F_{\b}$ equal to the inclusion map on $\pi^{-1} (Y_1)$ and to
 $ \b  \circ i_{\pi^{-1}(Y_2)}$ on $\pi^{-1}(Y_2)$, and extending equivariantly.
 In particular, it satisfies
\be\label{id} F_{\b}(\Gamma x)=\Ga_{\b} x,   \quad\text{for all } x\in X - \pi^{-1}(\Gamma \cdot \tilde\beta ).\ee 

For $c>0$, let $Y_{c}$ be the complement of the open $c$-neighborhood $\cal N_{c_0}(\Ga \cdot \tilde \beta)$ of $\Ga \cdot \tilde \beta$ in $Y$.
Set 
\begin{equation}
\label{def:Xc}X_{c}:= \pi^{-1}(Y_{c}).
\end{equation}
Both $Y_{c}$ and  $X_{c}$ are $\Gamma$-invariant.
Thus the restriction of $F_{\b}$ to $X_{c}$ descends to local isometry
$$
 f_{\b,c} : \Gamma\backslash X_{c} \to \Gamma_{\b}\backslash X. $$

In this section, we prove:

\begin{theorem}\label{thm:pi}  For all $c> \wdt(\b)$, the map
$$f_{\b,c}:\Gamma\backslash X_{c} \to \Gamma_{\b}\backslash X$$ 
is a proper locally isometric embedding, which is given by
$$f_{\b,c} ( [\Gamma x])= [\Ga_{\b} x]\quad\text{for all } x\in X_{c}.$$
\end{theorem}

The above discussion has an analogue when $X$ is replaced by $G$, as we now discuss: Consider the fibration $p: G \to X$ given by $g \mapsto go$, $\in G$, whose fibers are isomorphic to $K$. Define a $\rho_{\b}$-equivariant map
\[
 \bar F_{\b}: G - (\pi \circ p)^{-1}(\Gamma \cdot \tilde\beta) \;\to\; G
\]
as follows. In the separating case, set $\bar F_{\b}$ to be the identity on 
$(\pi \circ p)^{-1}(Y_1)$, the composition of the inclusion map with $\b$ on 
$(\pi \circ p)^{-1}(Y_2)$, and then extend uniquely by requiring equivariance. 
The construction is analogous in the non-separating case.
In particular,
\be\label{id2} \bar F_{\b}(\Gamma g)=\Ga_{\b} g\quad \text{for all $g\in G - (\pi\circ p)^{-1}(\Gamma \cdot \tilde\beta)$.}\ee 

For $c>0$, let $$G_c \coloneqq p^{-1}(X_{c}),$$ where $X_c$ is defined by \eqref{def:Xc}. 
The map $\bar F_\b$ descends to a local isometry
\[
 \bar f_{\b,c} : \Gamma \backslash G_c \to \Gamma_{\b} \backslash G.
\]
In this setting, \Cref{thm:pi} implies the following:

\begin{corollary}\label{piG}
For all $c>\wdt(\b)$, the map
     \[
     \bar f_{\b,c}: \Gamma \backslash G_c \to \Gamma_{\b} \backslash G
     \]
     is a proper locally isometric embedding, which is given by
$$\bar f_{\b,c} ( [\Gamma g])= [\Ga_{\b} g]\quad \text{ for all $g\in G_c$.}$$
\end{corollary}

\begin{proof}
    Consider the following commutative diagram:
    \[
    \begin{tikzcd}
    \Gamma \backslash G_{c} \arrow[r, "\bar f_{\b,c}"] \arrow[d, "p"'] & \Gamma_{\b} \backslash G \arrow[d, "p"] \\
    \Gamma \backslash X_{c} \arrow[r, "f_{\b,c}"] & \Gamma_{\b} \backslash X
    \end{tikzcd}
    \]
    The properness of $\bar f_{\b,c}$ follows from that of $f_{\b,c}$ (by \Cref{thm:pi}) since the fibers of the vertical maps in the above diagram are compact. 

    We show that $\bar f_{\b,c}$ is injective. Let $g_1, g_2 \in \Gamma \backslash G_{c}$ be distinct points. Suppose first that $p(g_1) \ne p(g_2)$. Since $f_{\b,c}$ is injective (by \Cref{thm:pi}), we have $f_{\b,c}(p(g_1)) \ne f_{\b,c}(p(g_2))$. The commutativity of the diagram then implies that $\bar f_{\b,c}(g_1) \ne \bar f_{\b,c}(g_2)$. If instead $p(g_1) = p(g_2)$, then the conclusion $\bar f_{\b,c}(g_1) \ne \bar f_{\b,c}(g_2)$ follows from the fact that $\bar F(s)$, and hence $\bar f_{\b,c}$, maps fibers isomorphically onto fibers of $p$.
\end{proof}

The rest of this section is devoted to the proof \Cref{thm:pi}.
Without loss of generality, we may assume that $\b\in A $ by conjugating $\Ga$.

\subsection*{Nearest point projection revisited}
The proof of Theorem \ref{thm:pi} is based on the study of fibers of the 
 the nearest point projection map $\pi : X \to Y$. Recall that 
  $\pi^{-1} (h o)=  h K_0  Bo$ for any $h\in H$.

Also, recall the notation $k(\theta)\in K_0 $ from \eqref{kt} and  $k_1\in K$ from $\eqref{k1}$.

\begin{lem}\label{q}
For all  $\theta,c\in \br$, there exists 
 $\theta'\in\R$ such that
$$ \,k_1^{-1} k(\theta')\,a_{c}\,k(\theta)\,k_{1} \in A .$$
\end{lem}

\proof
Let
$$
Q(\theta',\theta,c):=k_1^{-1}k(\theta')a_{c}k(\theta)k_{1}. $$
A direct multiplication yields $Q(\theta',\theta,c)$ is given by
$$
\tfrac12
\begin{pmatrix}
e^{c} & 0 & 0\\
0 &
  \tfrac{1}{e^{2c}}\bigl(\cos\theta\cos\theta'-e^{3c}\sin\theta\sin\theta'\bigr) &
   \tfrac{1}{e^{2c}} \bigl(\cos\theta\sin\theta'+e^{3c}\sin\theta\cos\theta'\bigr)\\
0 &
  \tfrac{-1}{e^{2c}}
  \bigl(\cos\theta\sin\theta'+e^{3c}\sin\theta\cos\theta'\bigr)  &
   \tfrac{1}{e^{2c}}\bigl(e^{3c}\cos\theta\cos\theta'-\sin\theta\sin\theta'\bigr)
\end{pmatrix}.
$$

If $\sin\theta=0$  then
set $\theta'=0$.
Else, if $\sin\theta\neq0$, choose $\theta'$ by
\[
\cot\theta'=-\,e^{-3c}\cot\theta.
\tag{3}
\]
Then  $Q(\theta',\theta,c)$ is diagonal. If the diagonal entry has a negative sign, we can replace $\theta'$ by $\theta' +\pi$ to make all diagonal entries of $Q(\theta', \theta,c)$ positive.
\qed

\begin{lem}\label{lem:Hdist}
  For $c\in\R$, the Hausdorff distance between $\pi^{-1} (o)$ and $a_c \pi^{-1} (o)$ is at most $|c|$. 
\end{lem}
\begin{proof} By \Cref{fiber2},
$$\pi^{-1}(o)= K_0  k_1 \cal W A^+ o .$$ 
Let $w\in \cal W$. By Lemma \ref{q}, for any $k\in K_0$, there exists $k'\in K_0$ such that $a_c k k_1\in k'k_1 A$. Since $A\subset wPw^{-1}$, we get  
$$ a_c k k_1 w P =k'k_1 wP  .$$
In other words, the Weyl chambers $a_c k k_1 wA^+o$ and $k' k_1 w A^+ o$ in $X$ are asymptotic.

Therefore by \cite[1.6.6(4)]{Eb},
$$d_{\rm Haus} (  a_c k k_1 wA^+o, k' k_1 w A^+ o )=d (a_c k k_1 w o, k'k_1 wo)= d( a_co, o)=|c|,$$
where $d_{\rm Haus}$ denotes the Hausdorff distance.
It follows that the Hausdorff distance between $\pi^{-1}(o)=K_0 k_1 \cal W A^+o$ and 
$a_c\pi^{-1}(o)=a_c K_0 k_1 \cal WA^+o$ is  at most $|c|$. 
\end{proof}

\begin{corollary}\label{hdxy} Let $c\in\R$.
 \begin{enumerate}
     \item The Hausdorff distance between $\pi^{-1} (\tilde \beta)$ and $\b \pi^{-1} (\tilde \beta) = a_c \pi^{-1} (\tilde \beta)$ is at most $|c|$.
     \item Suppose that $c\ge \wdt(\b)$. If $y$ and $y'$ are points in $Y - \overline{\cal N_{c}(\tilde \beta)}$  lying
     in distinct connected components, then the fibers $\pi^{-1}(y)$ and $\b \pi^{-1}(y')$ are disjoint. 
 \end{enumerate} 
\end{corollary}

\begin{proof}
(1). Using the $A_0$-equivariance of $\pi$, this follows from \Cref{lem:Hdist}.

(2). Since the nearest point projection from $X$ to a convex subset is $1$-Lipschitz, the minimal distance between any two fibers $\pi^{-1}(y)$ and $\pi^{-1}(y')$ is precisely $d(y,y')$. 

Write $Y$ as the union of two closed half-planes $Y_-$ and $Y_+$ sharing the common boundary $\tilde \beta$. Using the nearest point projection $\pi$, the symmetric space $X$ can be written as the union of the following connected smooth submanifolds with boundaries:
\[
 X_\pm \coloneqq \pi^{-1} (Y_\pm).
\]
Note that $\partial X_\pm = \pi^{-1}(\tilde \beta)$.
The interiors of $X_+$ and $X_-$ are  disjoint.

Let $y\in (Y_- -\overline{\cal N_{c}(\tilde \beta)})$ be an arbitrary point, where $c\ge \wdt(\b)$.
So, $d(y, \tilde \beta)=c+\e$ for some $\e>0$.  Then
\begin{align*}
 d(\b_s\pi^{-1}(y),\partial (\b_s X_+)) 
 &= d(\pi^{-1}(y),\pi^{-1}(\tilde\beta))\ge c +\e ,
 \quad s\in \R,
\end{align*}
where
$
 \b_s
$
is defined in \eqref{eqn:bs}.
By part (1), the Hausdorff distance between $\partial X_- = \pi^{-1}(\tilde \beta)$ and $\partial (\b_s X_+)=\b_s\pi^{-1}(\tilde \beta)$ for  $s\in[0,1]$ is at most $\wdt(\b)$.
Since $c\ge \wdt(\b)$, by the triangle inequality, it follows that 
 \[
  d(\b_s\pi^{-1}(y),\pi^{-1}(\tilde \beta) )
  \ge \e,
  \quad s\in[0,1].
 \]
 Thus  $\{\b_s\pi^{-1}(y):\ s\in[0,1]\}$ lies in the same connected component of $X- \pi^{-1}(\tilde \beta)$ (since the family is disjoint from $\pi^{-1}(\tilde \beta)$), which must be $X_-$ as $\b_0\pi^{-1}(y')=\pi^{-1}(y)\subset X_-$. Thus, $\b_1 \pi^{-1}(y)= \b \pi^{-1}(y)$ must be disjoint from any fiber $\pi^{-1}(y')$ contained in the interior of $X_+$. 
 
 The same conclusion holds for $y\in  (Y_+ -\overline{\cal N_{c}(\tilde \beta)})$.
\end{proof}

We are now in a position to finish the proof of \Cref{thm:pi}.

\begin{figure}
  \includegraphics [height=6cm]{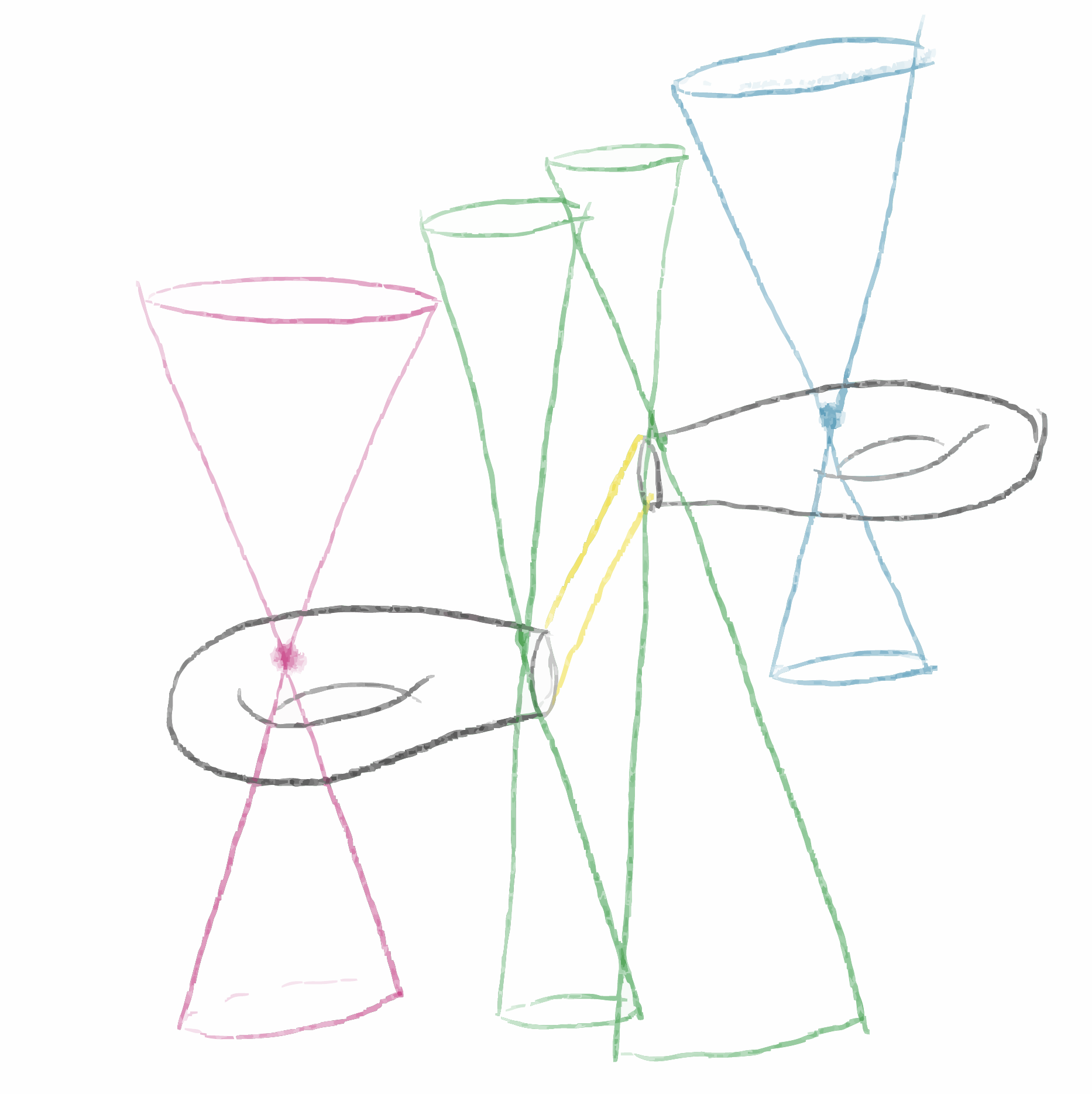}  
\caption{Fibers of the nearest projection map}
\end{figure}
\begin{proof}[Proof of \Cref{thm:pi}]
That $f_{\b,c}$ is a local isometry follows from the definition of the map. Thus, to show that $f_{\b,c}$ is injective, it is sufficient to show that $F_{\b}\vert_{X_{c}}$ is injective for all $c>\wdt(\b)$. 

Properness of $f_{\b,c}$ then follows: suppose, for contradiction, that there exists a divergent sequence $x_n$ in $\Gamma \backslash X_{c}$ such that the sequence $f_{\b,c}(x_n)$ converges to some point $z \in \Gamma_{\b} \backslash X$.
Being a local isometry, $f_{\b,c}$ is an open map in the interior of $\Ga\ba X_{\b}$ and so, by replacing $c$ by some $c'\in ( \wdt(\b),c)$, we have that $z$ belongs to the interior of $\op{Im} f_{\b,c'}$.
 As $f_{\b,c'}$ is injective, $f_{\b,c'}^{-1}$ is a continuous map in a sufficiently small neighborhood of $z$ contained in $\op{Im} f_{\b,c'}$. This implies that $x_n$ must converges to
 $f_{\b,c'}^{-1}(z)$, contradicting the assumption that $(x_n)$ is divergent.

Now we return to showing that $F_{\b}\vert_{X_{c}}$ is injective.
For this, we only need to show that for any two distinct connected components 
$
 Y',Y'' \subset Y -N_{c_0}(\Gamma \cdot\tilde\beta),
$
\begin{equation}\label{eqn:pi}
F_{\b_s}(\pi^{-1} (Y')) \cap F_{\b_s}(\pi^{-1} (Y'')) = \emptyset
\quad\text{for all }s\in[0,1].
\end{equation}
Since $\b_1 = \b$, the injectivity of $F_{\b}\vert_{X_{\b}}$ follows from this.

We will prove \eqref{eqn:pi} by induction on the Bass-Serre tree distance $d_T(v_{Y'}, v_{Y''})$ between the vertices corresponding to $Y'$ and $Y''$.

The base case when $d_T(v_{Y'},v_{Y''})=1$ follows  from \Cref{hdxy}(2).

For the inductive step, suppose that the assertion \eqref{eqn:pi} holds whenever $d_T(v_{Y'},v_{Y''})\le n_0-1$, for some $n_0\ge 2$. We show that the assertion still holds if $d_T(v_{Y'},v_{Y''})= n_0$. By the $\rho_{\b_s}$-equivariance, we may assume that $v_{Y'}$ is either $v_1$ or $v_2$, say $v_1$. Then, there is a geodesic sequence in $T$
\[
 \tilde v_1 = v_1, \tilde v_2,\dots, \tilde v_{n_0+1} = v_{Y''}
\]
connecting $v_1$ to $v_{Y''}$.
Let $L_2 \subset Y_{\tilde v_2}$ be a complete geodesic separating (in $Y$) $Y_1 = Y_{v_1}$ from $Y_{\tilde v_3}$, where $Y_{v}$ denotes the connected component of $ Y -N_{c_0}(\Gamma \cdot\tilde\beta)$ corresponding to the vertex $v\in T$. Clearly, $L_2$ also separates $Y_{1}$ from $Y_{\tilde v_4}\dots Y_{\tilde v_{n_0+1}}$. By our induction hypothesis, for all $s\in[0,1]$, the hypersurface $F_{\b_s}(\pi^{-1}(L_2))$ in $X$ does not intersect 
\[
 F_{\b_s}(\pi^{-1} (Y_{v_1})), \quad F_{\b_s}(\pi^{-1}  ( Y_{\tilde v_{n_0+1}})).
\] 
For $s=0$, i.e., for $F_{\b_0}$, the identity map,
the hypersurface $F_{\b_0}(\pi^{-1}(L_2))$ in $X$ separates $F_0(\pi^{-1} (Y_1)) $ from $F_0(\pi^{-1}  (Y_{\tilde v_{n_0+1}})) $. Therefore, by continuity, we conclude that in $X$, for all $s\in[0,1]$, $F_{\b_s}(\pi^{-1}(L_2))$ separates $F_{\b_s}(\pi^{-1} (Y_{v_1}))$ from $ F_{\b_s}(\pi^{-1}  (Y_{\tilde v_{n_0+1}}))$.
This proves \eqref{eqn:pi}, thereby concluding the proof of the result.
\end{proof}

\subsection*{Chaotic floating planes}

\begin{theorem}\label{ff} Let $ L=hA_0 o \in Y$ be a geodesic so that
$\Gamma\ba \Gamma L$ is disjoint from the $2r$-neighborhood of $\beta$ for some $r>0$.
There exists $t_0>0$, depending only on $L$ and $r$  such that for all $t>t_0$ and any $\b\in B$ with width $\wdt(\b) < r$, we have  
\begin{enumerate}
    \item 
$\dim \overline{ \Ga_{\b} \ba \Ga_{\b} h a_t H }=\dim \overline{\Gamma\ba \Gamma h A_0} +2 ;$
\item If ${\Gamma hA_0}$ is admissible,
then $$\frac{1}{2}\left( \dim \overline{\Gamma\ba \Gamma h A_0} +3 \right) \le  \overline{ \Ga_\b \ba \Ga_\b Y_{L,t}} \le \dim \overline{\Gamma\ba \Gamma h A_0} +1 .$$
\item  Let $L$ be an admissible geodesic with $1<\dim (\overline{\Gamma\ba \Gamma L}) <2$. Then
    $$\frac{1}{2}\left( \dim \overline{\Gamma\ba \Gamma L} +3 \right) \le  \overline{ \Ga_\b \ba \Ga_\b Y_{L,t}} \le \dim \overline{\Gamma\ba \Gamma L} +1 .$$ 
   In particular, $$2< \dim \overline{\Ga_\b Y_{L,t}} <3.$$     
\end{enumerate}

\end{theorem}
\begin{proof}
    Let $L:=hA_0o$. Since $\beta\subset S=\Gamma\ba Y$ is closed,  $\{ x\in S:\  d(x, \beta)\le r\}$ is closed and hence $\overline{\Gamma L}$ is disjoint from $N_{r_0}(\tilde \beta)$ for some $r_0$ strictly bigger than $\wdt(\b)$. Now consider the floating plane $Y_{L, t}=
    h a_t H o$. By Theorem \ref{narrow}, there exists $t_0>0$ depending on $L$ and $r$ such that for all $t\ge t_0$,
   the nearest projection $\pi (Y_{L, t})$ is contained in the $r/2$-neighborhood of $L$. By the $H$-equivariance, the nearest projection $\pi (\ga Y_{L, t})$ is contained in the $r/2$-neighborhood of $\ga L$ for all $\ga\in \Gamma$. Therefore
   $\pi(\overline{\Gamma ha_tHo})$ is disjoint from $r$-neighborhood of $\tilde \beta$.
   Hence,   $$\overline{\Gamma\ba \Gamma ha_tH} \subset \Gamma\ba G_{r} .$$

    Since $\wdt(\b) < r$,  Corollary \ref{piG} implies that
     \[
     \bar f_{\b,r}: \Gamma \backslash G_{r} \to \Gamma_{\b} \backslash G
     \]
     is a proper locally isometric embedding, where $\bar f_{\b,r}([\Gamma g])=[\Ga_\b g]$ for all
     $g\in G_{\b}$. The local isometric property everywhere implies that the Hausdorff dimension of 
     $\overline{\Gamma\ba \Gamma ha_tH}$ is equal to that of its image in $\Ga_{\b}\ba G$. Since  $ \bar f_{\b,r} (\Gamma\ba \Gamma g) =\Gamma_{\b}\ba \Gamma_{\b} g$ and $\bar{f}_{\b,r}$ is a proper map, the image of  $\overline{\Gamma \ba \Gamma ha_t H}$ is closed and hence 
     $$ \bar f_{\b,r}( \overline{\Gamma \ba \Gamma ha_t H})= \overline{\Gamma_{\b} \ba \Gamma_{\b} ha_t H} .$$
     
     Therefore the claim (1) follows from \Cref{hd2}. Similarly, \Cref{pi,dim} imply (2).  If $1<\dim (\overline{\Gamma\ba \Gamma L}) <2$, then $\dim (\overline{\Gamma\ba \Gamma L})=
     \dim (\overline{\Gamma\ba \Gamma hA_0})$; see \Cref{ll} below. Hence (3) follows.
\end{proof}

\begin{Rmk}
 Much of the discussion in this paper also applies when $\Gamma < H$ is a torsion-free nonuniform lattice. In this situation, the quotient $\Gamma \backslash Y$ is a noncompact, finite-area hyperbolic surface with finitely many cusps. Choosing a nonperipheral simple closed curve $\beta \subset S$, one may again decompose $\Gamma$ over the cyclic subgroup generated by $\delta = [\beta] \in \Gamma$, as before. For each $\mathbf{b} \in B = C_G(\delta)^\circ$, we obtain a homomorphism $\rho_{\mathbf{b}} : \Gamma \to G$ defined as above.

For each $\mathbf b\in B$, $\rho_{\mathbf{b}}$ is discrete and faithful (\cite{FG}). 
In particular, \Cref{ma} remains valid in this setting.
\end{Rmk}

\begin{Rmk} For a cocompact lattice $\Ga<H$, Pavez showed that any Hausdorff dimension between $1$ and $3$ can occur as the dimension of the closure of some orbit $\Ga \ba \Ga h A_0\subset \Ga\ba H $ \cite{Pa};
any number between 3 and $5$ can arise as the Hausdorff dimension of the closure of a floating plane $\Ga\ba \Ga h a_t H$.

It is a natural question, for a given simple closed geodesic $\beta$, which Hausdorff dimensions can be realized by geodesic flow closures $\overline{\Ga \ba \Ga h A_0}\subset \Ga\ba H$ whose projection to the surface $\Ga\ba Y$ is disjoint from $\beta$. The maximum possible value in this setting is $1+2\delta_0$ where $\delta_0$ denotes the maximum of the critical exponents of the components of $\Ga\ba Y-\{\beta\}$. Thus the precise question is whether any number between $1$ and $1+2\delta_0$ can indeed be achieved.

In the next section, we show that we can find
 geodesic-flow closures that come arbitrarily close to dimension $1$, yet avoid a prescribed closed simple geodesic.
\end{Rmk}

\section{Geodesic closures away from a given simple closed geodesic}
 
 The goal of this subsection is to prove Theorem \ref{pair} and finish the proofs of
 all theorems in the introduction.

In the whole section, let $S=\Gamma\ba \bH^2$ be a closed hyperbolic surface for a torsion-free cocompact lattice $\Gamma<\PSL_2(\br)$. 
The following theorem combined with  with Theorem \ref{ff} implies Theorem \ref{ma} and \ref{ma3}.
 \begin{theorem}\label{thm:chaotic} \label{pair} Let $\beta_1, \dots, \beta_m$ be pairwise disjoint simple closed geodesics in $S$ and let $S'\subset S$ be a connected component of $S-\bigcup_i \beta_i$.
     There exists a sequence of immersed complete admissible geodesics $\ell_n \subset S$, contained in $S'$, such that the following hold:
     \begin{enumerate}
         \item $\inf\{ d(\ell_n , \bigcup_{k=1}^m \beta_k) :\ n\in\N\} >0$.
         \item For all $n\in\N$, $\dim \overline{\ell_n} >1$.
         \item $\dim \overline{\ell_n} \to 1$ as $n\to\infty$.
     \end{enumerate}
 \end{theorem}

Let $A_0$ be the diagonal subgroup of $\PSL_2(\br)$, $K_0=\SO(2)$ and $o=[K_0]=\PSL_2(\br)/K_0\simeq \bH^2$.
The unit tangent bundle of $\bH^2$ is $\PSL_2(\br)$ and an orbit of the geodesic flow is of the form $hA_0\subset \PSL_2(\br)$.  Let $p:\Ga\ba \PSL_2(\br)\to \Ga\ba \bH^2$ be the basepoint projection $x\mapsto xo$. 

We recall the following theorem of
Ledrappier and Lindenstrauss: for a Borel measure $\sigma$ on a metric space, we denote by $\ldim\sigma$ the lower information dimension $$\ldim\sigma=\text{\rm ess-inf}_x \left({\liminf}_{\e\to 0}\frac{\log \sigma(B(x, \e))}{\log \e}\right).$$
\begin{theorem}[{\cite[Theorem ~1.1]{LL}}] \label{ll0} Let $\mu$ be an $A$-invariant probability measure on $\Ga\ba \PSL_2(\br)$.
If $\ldim\mu \le 2$, then
  $$ \ldim\mu =\ldim p_*\mu .$$
\end{theorem}

In order to deduce the comparison between the Hausdorff dimension of an $A_0$-orbit and its projection to $S$ from Theorem \ref{ll0},
we first note a general principle: for any finite Borel measure $\sigma$ on a metric space, 
its lower information dimension satisfies 
\be\label{low} \ldim\sigma \le  \dim (\text{supp }\sigma).\ee 
Indeed, if $\alpha< \ldim \sigma$, then for almost all $x$, there exists $r_x>0$ such that
${\sigma(B(x, r))}\le {r^\alpha}$ for all $0<r<r_x$. By the mass distribution principle, this implies that the $\alpha$-Hausdorff measure of $\supp \sigma$ is positive, and hence $\dim (\supp \sigma)\ge \alpha$. Letting $\alpha\to \ldim\sigma$ yields \eqref{low}.

\begin{cor} \label{ll} Let
$\overline{xA_0}\subset \Ga\ba \PSL_2(\br)$ be the support of an $A_0$-invariant probability measure $\mu$. Suppose that $\ldim\mu=\dim \overline{xA_0}$ and $ \dim \overline{xA_0}\le 2 $. Then
  $$  \dim \overline{xA_0 o} = \dim \overline{xA_0 }. $$
\end{cor}
\begin{proof}
Applying \eqref{low} to the measure $\sigma=p_*\mu$ gives
$\ldim p_*\mu  \le  \dim (\overline{xA_0o})$.
Since
the projection $p$ is $1$-Lipschitz, Hausdorff dimension can only decrease under $p$. Hence 
$$ \dim \overline{xA_0 o}\le \dim \overline{xA_0} =\ldim\mu .$$
By Theorem \ref{ll0}, we have $\ldim p_*\mu=\ldim \mu$ and the two inequalities above therefore give
$\dim \overline{xA_0 o}= \dim \overline{xA_0}$.
\end{proof}

 For a non-elementary convex cocompact subgroup $\Ga_0<\PSL_2(\br)$,  
let $\Omega_{\Ga_0}\subset \Ga_0\ba \PSL_2(\br)=\op{T}^1(\Ga_0\ba \bH^2)$ denote the non-wandering set of the geodesic flow, which is the union of all $v\in \op{T}^1(\Ga_0\ba \bH^2)$ whose forward and backward end points of the geodesic determined by $v$ belong to the limit set of $\Ga_0$.  Let $\m=\m_{\Ga_0}$ denote
  the Bowen-Margulis-Sullivan measure on $\Omega_{\Ga_0}$, which is the $A_0$-invariant probability measure of maximal entropy \cite{Su}.  
  
 \begin{prop}\label{ahl}
     Let $\Ga_0<\PSL_2(\br)$ be a non-elementary convex cocompact subgroup.
     Then  $$\ldim \, \m_{\Ga_0}=\dim \Omega_{\Ga_0}= 1+2\delta_{\Ga_0}$$
     where $\delta_{\Ga_0}$ is the critical exponent of $\Ga_0$. Moreover, $\Omega_{\Ga_0}$ is admissible in the sense of Definition \ref{adm}. \end{prop}
\begin{proof} Let $\La_0$ denote the limit set of $\Ga_0$ and set $\delta_0=\delta_{\Ga_0}$.
By Sullivan \cite{Su}, the Patterson-Sullivan measure on $\La_0$
is proportional to the Hausdorff measure $\cal H^{\delta_0}|_{\La_0}$  and $\La_0$ is Ahlfors $\delta_0$-regular:  there exists $c>1$ such that for any $\xi\in \La_0$ and $0<r\le  \text{diam}(\La_0)$, 
 $$c^{-1} \; r^{\delta_0}\le \cal H^{\delta_0} (B(\xi, r)\cap \La_0)\le c \; r^{\delta_0} .$$
 In particular, any nonempty open subset of $\La_0$ has both Hausdorff and box dimension equal to  $\delta_0$, so (see, e.g., \cite{Fa}),
 $$\dim (\La_0\times \La_0) = 2\dim \La_0=2\delta_{0}$$

   The Hopf parametrization $$\Phi: (\partial \bH^2 \times \partial \bH^2 -  \text{diag})\times \br \to \T^1(\bH^2)$$ is locally bi-Lipschitz and
 $\Omega_{\Ga_0}= \Ga_0\ba \Phi ((\La_0\times \La_0 -  \text{diag})  \times \br)$.
Hence
     $$\dim \Omega_{\Ga_0}= 1+\dim (\La_0\times \La_0 -  \text{diag} ) =1+2\delta_0 .$$
   Since products and locally bi-Lipschitz images of Ahlfors regular sets are Ahlfors regular, it follows from the $\delta_0$-Ahlfors regularity of $\La_0$ that $\Omega_{\Ga_0}$ is $1+2\delta_0$-Ahlfors regular.
 Moreover, in Hopf coordinates, the BMS measure $\m_{\Ga_0}$ is locally equivalent to $\cal H^{\delta_0}|_{\La_0}\times \cal H^{\delta_0}|_{\La_0} \times \op{Leb}$,  so the push-forward to $\Omega_{\Ga_0}$ is locally equivalent to $\cal H^{1+2\delta_0}|_{\Omega_{\Ga_0}}$. Since   $\Omega_{\Ga_0}$ is $1+2\delta_0$-Ahlfors regular, it follows that
 $$\ldim \, \m_{\Ga_0}=1+2\delta_0.$$
Finally,  admissibility follows since,  for all sufficiently small basic open subsets $\cal O$,   $\pi_{\pm}(\Omega_{\Ga_0} \cap \cal O)
     $ are  open subsets of $\La_0$
\end{proof}

 Before discussing the proof of Theorem \ref{pair}, we describe a procedure for constructing immersed geodesics $\ell$ in a compact hyperbolic surface $\Sigma'$ with nonempty geodesic boundary $\partial \Sigma'$ such that $\dim \overline \ell >1$.

 Suppose that $ \Sigma'$ is a convex compact surface of $S$ whose boundary is the disjoint union of simple closed geodesics $\alpha_1, \dots, \alpha_k$.
Let $\Sigma$ denote the double of $\Sigma'$ along $\partial \Sigma'$; this is a closed hyperbolic surface.
 Consider the universal covering map $\bH^2\to \Sigma$, and identify $\pi_1(\Sigma)$ with a cocompact lattice in $\PSL(2,\R)$ via the holonomy representation.
For each $i$, let $\tilde \alpha_i\subset \bH^2$ be a lift of $\alpha_i$ to $\bH^2$. Let $U$ be a connected component of $\bH^2-\bigcup_{i=1}^k \pi_1(\Sigma) \tilde \alpha_i$ so that the projection of $U$ to $\Sigma$ is precisely $\Sigma'$. Let $\Ga_0$ be the stabilizer of $U$ in $\pi_1(\Sigma)$. 
 Then $$\Sigma_0:=\Ga_0\ba \bH^2$$ is a convex cocompact hyperbolic surface whose compact convex core is isometric to $\Gamma_0\ba \overline{U} = \Sigma'$. The identity map of $\bH^2$ induces a covering map
 $$f:\Sigma_0\to \Sigma $$
and its restriction to the convex core of $\Sigma_0$ is an isometry onto $\Sigma'$. We refer to the critical exponent of $\Ga_0$ as the critical exponent of $\Sigma_0$, or of $\Sigma'$ by abuse of terminology.

\begin{prop}\label{bms} Let $\Sigma'$ be a convex compact subsurface of $S$ whose boundary 
$\partial \Sigma'$ is a disjoint union of simple closed geodesics.
If the critical exponent $\Sigma'$, denoted by $\delta'$, is strictly smaller than $1/2$,
   then there exists a complete admissible geodesic $\ell \subset \Sigma'$ 
  whose closure $\overline \ell\subset \Sigma'$ has Hausdorff dimension $1+2\delta'$.
\end{prop}

\begin{proof}
Let $\Sigma_0$ be as above, and
  consider its unit tangent bundle $\T^1(\Sigma_0)=\Ga_0\ba \PSL_2(\br)$.   By \Cref{ahl}, the non-wandering set $\Omega_0$ is admissible and
  \be\label{d1} \dim \Omega_0= 1+2\delta '=\ldim  \mathsf m_{\Ga_0}\ee 
where $\mathsf m_{\Ga_0}$ denotes the Bowen-Margulis-Sullivan measure on $\Ga_0\ba \PSL_2(\br)$.
  Since $\m_{\Ga_0}$ is $A_0$-ergodic,  $\mathsf m_{\Ga_0}$-almost all geodesic flow lines are dense in $\Omega_0$. In particular, there exists a geodesic flow line $\cal G\subset \T^1(\Sigma_0)$ such that 
  \be\label{d2} \overline{\cal G}=\Omega_0 =\supp \mathsf m_{\Ga_0}.\ee 
The basepoint projection $p:\T^1(\Sigma_0)\to \Sigma_0$ maps
$\Omega_0$
into the convex core of $\Sigma_0$ which embeds isometrically into $\Sigma'$. Hence  $p(\cal G) \subset {\Sigma'}$, and we may regard $\cal G$ as a geodesic in $\T^1(\Sigma)$ with basepoints in $\Sigma'$. Since $\delta'<1/2$, we have $\dim \overline{\cal G} <2$.
Combining \eqref{d1} and \eqref{d2}, we may apply \Cref{ll}  to $\cal G\subset \T^1(\Sigma)$ and deduce
$$\dim 
\overline{p(\cal G)} =\dim \overline{\cal G} .$$  Since $\overline{\cal G}=\Omega_0$, ${\cal G}$ is admissible. Setting $\ell=p(\cal G)$, we obtain $\dim \overline \ell=1+2\delta'$, completing the proof.
\end{proof}

To construct immersed geodesics with closure of dimension just above $1$,  we will produce convex cocompact subsurfaces with arbitrarily small critical exponent. The following elementary observation provides a convenient source of such examples.
\begin{lem}\label{zero}
 Let $g_1, g_2$ be hyperbolic elements of $\PSL_2(\br)$ which generate a Schottky subgroup and set $\Ga_n:=\langle g_1^n, g_2^n \rangle$ for $n\in \N$. Then
 the critical exponent $\delta_{\Ga_n}$ tends to $ 0$ as $n\to \infty$.
\end{lem}
\begin{proof} Since the word metric on $\Ga_1$ and the restriction of the hyperbolic metric on the orbit $\Ga_1 o$ are quasi-isometric,
there exists a constant $c>1$ (independent of $n$) such that  any reduced word $w$ in $g_1^{\pm n}$ and $g_2^{\pm n}$ of length $k$ satisfies $d(o, wo)\ge c^{-1} nk -c$. Since the number of reduced words of length $k$ in $\Z*\Z$
is $4\cdot 3^{k-1}$, there exists $c'>0$ such that the number of $w\in \Ga_n$ with $d(o, wo)<T$
is at most $c'\cdot e^{ c' T/n}$ for all $n\ge 1$. Since
$$\delta_{\Ga_n}=\limsup_{T\to \infty}\frac {1}{T}\log \#\{w\in \Ga_n: d(o, wo)<T\}\le c'/n,$$  the claim follows. \end{proof}

In particular, by replacing a Schottky subgroup with large powers of its generators, we obtain convex cocompact surfaces whose limit sets have dimension approaching zero. These will serve as building blocks for the admissible geodesics we construct inside a pair of pants.

We now explain how Lemma \ref{zero} may be combined with Proposition \ref{bms} to produce admissible geodesics inside a fixed pair of pants $S'\subset S$ whose closures have Hausdorff dimension just slightly larger than $1$ but uniformly bounded away from the boundary of $S'$.
\begin{proposition}\label{pant}
    Let $S'\subset S$ be a pair of pant with geodesic boundary. Then there exists a sequence of admissible geodesics $\ell_n$ such that $\overline{\ell_n}\subset \op{int} S'$, $\inf_{n\in \N} d(\partial S',\overline{\ell_n})>0$, $\dim \overline{\ell_n}>1$ for all $n\in\N$, and
    $$\dim \overline{\ell_n} \to 1\quad\text{ as $n\to \infty$.}$$
\end{proposition}
\begin{proof}
     Let $\Gamma_0<\Ga $ be a convex cocompact subgroup such that the core of the surface
\[
  \Sigma_0 \coloneqq \Gamma_0 \backslash \mathbb{H}^2
\]
is isometric to $S'$. We will identify the core of $\Sigma_0$ with $S'$.
Choose generators $\gamma_1,\gamma_2 \in \Gamma_0$ such that the boundary curves of $S'$
are represented by $\gamma_1, \, \gamma_2, \, \gamma_1\gamma_2 \in \Gamma_0$.

Our goal is to find, inside $S'$, a family of convex cocompact coverings $\Sigma_n$ whose critical exponents $\delta_{\Ga_n}$ tend to zero. Proposition \ref{bms} will then guarantee the existence of admissible geodesics on each $\Sigma_n$ with closure of dimension $1+2\delta_{\Ga_n}$.

Pick a pair of hyperbolic elements $g_1,g_2 \in \Gamma_0$ which generates a Schottky subgroup that
contains no nontrivial powers of conjugates of the elements $\gamma_1,\,\gamma_2,\,\gamma_1\gamma_2$.
To see such $g_1,g_2 \in \Gamma_0$ exist, put a hyperbolic structure on $\Sigma_0$ such that all its ends
are cusps. In this case the only parabolic elements of $\Gamma_0$ are the nontrivial powers of conjugates
of $\gamma_1,\,\gamma_2,\,\gamma_1\gamma_2$. But since $\Gamma_0$ is non-elementary, it contains a convex cocompact Schottky group generated by $g_1$ and $g_2$ which then cannot contain any such elements.

    For each $n\in\N$, set
    \[
     \Gamma_n \coloneqq \langle g_1^n,g_2^n\rangle ,\quad \Sigma_n \coloneqq \Gamma_n\backslash \bH^2.
    \]
 By Lemma \ref{zero},
    \[
     \delta_{\Gamma_n} \to 0 \quad \text{as }n\to\infty.
    \]
Consider the covering map 
    \[
     p_n : \Sigma_n \to S
    \]
    which clearly factors through $p_1: \Sigma_1\to S$. 
    
 We claim that the image $p_1(\operatorname{core}(\Sigma_1))$ misses the boundary of $S'$.
To see this, suppose for contradiction that there exists $x \in \operatorname{core}(\Sigma_1)$ such that $p_1(x) \in c'$, where $c' \subset \partial S'$ is a connected component. Since $p_1$ is a local isometry and $p_1(\operatorname{core}(\Sigma_1)) \subset S'$, it follows that $x \in \partial\operatorname{core}(\Sigma_1)$.
Let $c \subset \partial\operatorname{core}(\Sigma_1)$ be the connected component containing $x$. Then $p_1(c)$ is a closed geodesic contained in $S'$ containing $p_1(x)$. On the other hand, since $p_1(x)\in \partial S'$, 
$c'$ is the only closed geodesic passing through
$p_1(x)$ contained in $S'$. Therefore 
 $p_1(c) = c'$.
Since $c' \subset S'$ is peripheral in $S'$, the conjugacy class in $\Gamma_1$ representing $c$ contains a conjugate of a nontrivial power of $\gamma_1$, $\gamma_2$, or $\gamma_1 \gamma_2$, contradicting the choice of $g_1, g_2$ made above.

 Since $\delta_{\Gamma_n}\to0$ as $n\to\infty$, \Cref{bms} guarantees that for all large enough $n\in\N$, there exists a geodesic  $\ell'_n$ lying in the interior of $\op{core}(\Sigma_n)$ such that $\overline{\ell_n'}$ has Hausdorff dimension $1+2 \delta_{\Gamma_n}$. 
    Let
    \[
     \ell_n = p_n(\ell'_n) \subset p(\core(\Sigma_1)).
    \]
    Since $\core(\Sigma_n)$ is compact and hence the restriction of $ p_n$ to $\core(\Sigma_n)$ is a proper immersion, it follows that 
    \[
     \overline{\ell_n} = p_n(\overline{\ell'_n}) 
     \quad\text{and}\quad
     \dim {\overline{\ell_n}} = \dim \overline{\ell'_n} = 1+2 \delta_{\Gamma_n}
    \]
    for all large $n\in\N$. This finishes the proof.
\end{proof}
Now we are ready for the proof of \Cref{thm:chaotic}.

\begin{proof}[Proof of \Cref{thm:chaotic}]
    We can extend the set of simple closed geodesics $\{\beta_1, \dots, \beta_m \} \subset S$ to pants decomposition of $S$ whose boundaries are geodesics. Let $S'\subset S$ be a pair of geodesic pant in this decomposition and $\ell_n$ be the sequence of admissible geodesics given by \Cref{pant}. Since $\overline{\ell_n}$ lies in the interior of $S'$, they do not intersect any of $\beta_i$, as desired.
\end{proof}

\subsection*{Proof of Theorem \ref{max}} It follows from  Theorem \ref{ma} except for the integrability claim.

By \cite{LT}, 
there exists a closed hyperbolic surface $S=\Ga\ba \bH^2$, where $\Ga<H$ is a cocompact lattice, such that the inclusion homomorphism $\varphi: \Gamma \to G$, 
\[
 \Gamma  \to  H \hookrightarrow G
\]
is {\em integral}, i.e., $\varphi(\Gamma) <\SL_3(\Z)$. Let $\beta$ be an oriented simple closed geodesic in $S$.
Choosing a basepoint $x_0$ on $\beta$, we identify $\pi_1(S,x_0)$ with $\Gamma$. Let $\gamma = [\beta] \in \Gamma$ be the element representing $\beta$. We can further assume that $\beta\in A_0 = A\cap H$. Suppose that there also exists an element $\mathsf a\in (A- A_0)$ such that the bulged representation
$\rho_{\mathsf a}$ is also integral. See \cite{LT} for such examples.

Pick an auxiliary oriented non-separating simple closed geodesic $\sigma \subset S$ disjoint from $\beta$ and extend $\sigma\cup \beta$ it to a geodesic pants decomposition of $S$ so that there exists a pair of pant $S_0\subset S$ disjoint from $\sigma$.
By \Cref{thm:chaotic}, there exists an immersed complete
geodesic $\ell = \Gamma \backslash \Gamma L \subset S$ contained in $S_0$ such that 
\[
 d(\overline{\ell}, \beta \cup \sigma) > 0 
 \quad \text{and} \quad 
 1 < \dim \overline{\ell} < 2.
\]

For any $r>0$, there exists a finite Riemannian cover
\begin{equation}\label{eqn:cover}
 p: S' \to S
\end{equation}
such that the following holds: there are connected components $\beta'$ of $p^{-1}(\beta)$ and a lift $\ell'$ of $\ell$ in $S'$ such that
\[
 p\vert_{\beta'} \quad\text{and}\quad  p\vert_{\overline{\ell'}}
\]
are both homeomorphisms onto $\beta$ and $\bar\ell$, respectively, and 
\begin{equation}\label{eqn:r}
 d(\overline{\ell'},\beta') > r.
\end{equation}

One can construct such a cover \eqref{eqn:cover} as follows. Let $\mathsf g$ be the genus of $S$.
Since $\sigma$ is non-separating, the abelianization map
\[
 \Gamma \to \Gamma^{\rm ab} \cong \Z^{2\mathsf g}
\]
maps the homotopy class $[\sigma]$ to a nontrivial primitive element $z\in \Gamma^{\rm ab}$. Fix a homomorphism $\Gamma^{\rm ab} \to \langle z\rangle \cong \Z$, extending the identity map $\langle z\rangle\to \langle z\rangle$, and consider its composition with $\Gamma \to \Gamma^{\rm ab}$; its mod $n$ reduction gives a surjection
\[
 f_n : \Gamma \to \Z/n\Z.
\]
 Consider the degree $n$ regular Riemannian covering
\[
 p_n : S_n=\ker f_n\ba \bH^2 \to S.
\]
Each $p_n^{-1}(\overline{\ell})$ and $p_n^{-1}(\beta)$  has $n$ connected components, each homeomorphic to $\overline{\ell}$ and $\beta$, respectively, under the covering map $p_n$. Picking $n$ large enough, one may choose appropriate connected components $\bar\ell'$ and $\beta'$ of $p_n^{-1}(\overline{\ell})$ and $p_n^{-1}(\beta)$ so that \eqref{eqn:r} holds. Compare with the figure below:

\begin{figure}[ht]
\centering
\begin{overpic}[width=0.7\textwidth,tics=5]{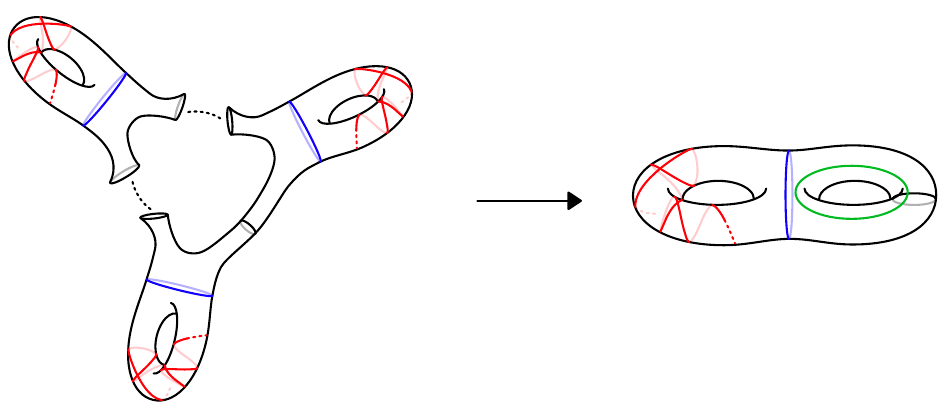} 
\put(18,14) {\color{blue}\tiny $\beta'$}
\put(80.5,21) {\color{blue}\tiny $\beta$}
\put(4,33.5) {\color{red}\tiny  $\ell'$}
\put(68,22) {\color{red}\tiny  $\ell$}
\put(89.5,18) {\color{Green}\tiny $\sigma$}
\put(7,19) {\small $S_n$}
\put(82,13) {\small $S$}
\put(55,24) {\scriptsize $p_n$}
\end{overpic}
\caption{}
\end{figure}

Since $p$ is a proper immersion, one has $$\dim {\overline \ell} = \dim{\overline{\ell'}}.$$
Note that $\mathsf a \in a_{s_0} A_0$, where $s_0 = \wdt(\mathsf a)$. Choose an appropriate covering $p$ in \eqref{eqn:cover} such that $r$ given by \eqref{eqn:r} is strictly greater than $|s_0|$. Equip $S'$ with the basepoint $x_0' = p\vert_{\beta'}^{-1}(x_0)$. Then, the monomorphism
\[
 p_* : \pi_1(S',x_0')\longrightarrow \pi_1(S,x_0) = \Gamma
\]
maps the homotopy class of $\beta'$ to $\gamma$. Let $\Gamma' \coloneqq \rho(\pi_1(S',x_0'))$. Clearly, the representation
\[
 \rho'_{\mathsf a} \coloneqq (\rho_{\mathsf a})|_{\Gamma'}
\]
is integral and $\rho_{\mathsf a}(\Ga') <\SL_3(\Z)$ satisfies the hypothesis of \Cref{ma}.

\appendix
\section{Orthogonal planes with fractal closures}
In this appendix, we describe the closures of geodesic planes which are orthogonal to $Y$ along a geodesic $L$. Using the bulging deformation,
we show that  closures of an orthogonal plane can be as chaotic as the closure of $L
$ (Theorem \ref{ort}).
By an irreducible geodesic plane in $X$, we mean a totally geodesic plane of the form $gY=g H o\subset X$ for some $g\in G$.
\begin{lemma}
   For every complete geodesic line $L\subset Y$, there is a unique irreducible geodesic plane $Z_L$ in $X$ such that $Y \cap Z_L = L$. Moreover, $Y$ and $Z_L$ are orthogonal to each other.
\end{lemma}
\begin{proof}
  Without loss of generality, we may assume that $L =A_0 o= \{ h_t o :\ t\in\R\}$. Let $Z= gY$, $g\in G$, be an irreducible geodesic plane containing $L$. Since $Z$ contains $o$, we may assume $g=k\in K$, replacing $g$ by $gh$ for some suitable $h\in H$. Now $Z=k H k^{-1}(o)\supset A_0 o $. By considering the tangent subspaces, it implies that
  $\fa_0\subset k \frak h k^{-1}$.
Therefore
  $k^{-1}\fa_0 k$ is a one-dimensional symmetric subspace of $\fh$ and hence must be of the form $k_0^{-1}\fa_0 k_0$ for some $k_0\in K_0$. Hence by replacing $k$ by $kk_0^{-1}$,
  we may assume that $ k\in N_K(\fa_0)$.
  By a direct computation, we can show that $N_K(\fa_0)$
  is the subgroup generated by $A\cap K=\{g_i: 1\le i\le 4\}$ and 
  $w_0 = \begin{pmatrix}
            0 & 0 & 1 \\ 0 & -1  & 0 \\ 1 & 0& 0
        \end{pmatrix}$ where
        $g_1=e$, $g_2=\diag(-1,1,-1)$, $g_3=\diag(1,-1,-1)$ and $g_4=(-1,-1,1)$.
    Since $g_2$ and $w_0$ normalizes $H$, and $g_3=g_4 g_2$,
    we may assume that $g=g_3$ up to the normalizer of $H$.
    Therefore $k Y=g_3 Y $ is the only irreducible geodesic plane containing $L$, which is different from $Y$.
We set $Z_L:=g_3Y$. Moreover, $Y$ and $Z_L$
are orthogonal to each other as can be checked using the Killing form. This proves the claim.
\end{proof}

The following result is an analog of \Cref{dim} in the case of orthogonal planes:

\begin{lemma}
    Let $\Gamma <H $ be a discrete subgroup and let $L\subset Y$ be a complete geodesic with $\dim \overline{\Gamma L} < 2$. Then
    \begin{equation} \label{eq:dim-ineq}
 \dim \bigl(\overline{\Gamma Z_L}\bigr) = \dim \bigl(\overline{\Gamma L}\bigr) + 1.
\end{equation}
\end{lemma}

\begin{proof}
Write
$ \overline{\Gamma Z_L} = \bigcup_l Z_{l}$
    where the union is taken over all complete geodesics $l\subset \overline{\Gamma L}$.
Consider the normal bundle $N_YX$ of $Y \subset X$.  
The exponential map
\[
\exp : N_YX \longrightarrow X
\]
is a diffeomorphism.
Now, $\overline{\Gamma Z_L}$ can be written as the union of a family $\mathcal{L}$ of geodesic lines in $X$ orthogonal to $Y$, where any two such lines are allowed to intersect only inside $Y$. Correspondingly, there is a family of lines $\mathcal{L}' \subset N_YX$ whose image under the exponential map is precisely $\mathcal{L}$. Since the exponential map is a diffeomorphism, the Hausdorff dimension of $\overline{\Gamma Z_L}$ is precisely $\dim \bigl(\bigcup_{l \in \mathcal{L}'} l \bigr)$.
Denote by $\mathbb{P}(N_YX)$ the projectivization of $N_YX$. Then the family $\mathcal{L}'$ determines a closed subset  $R \subset \mathbb{P}(N_YX).$
Since the natural projection
\[
p : \mathbb{P}(N_YX) \longrightarrow Y
\]
is a submersion and satisfies $p(R) = \overline{\Gamma L}$, we obtain
$\dim R \ge \dim \bigl(\overline{\Gamma L}\bigr).$
On the other hand, the union of the lines satisfies
$$
 \dim \overline{\Ga Z_L}=\dim \Bigl(\bigcup_{l \in \mathcal{L}'} l \Bigr) = \dim R + 1.$$
Therefore 
$ \dim \overline{\Ga Z_L} \ge  \dim \bigl(\overline{\Gamma L}\bigr) + 1.  $
To prove the reverse inequality, note that
on each basic open subset $\cal O=h_0U^+U^-A_0$ of $G$ where $U^{\pm}\subset N_0^{\pm}$ and $h_0\in H$,
we have 
$\overline{\Gamma A_0} \cap \cal O = C_{\mathcal{O}} \times A_0,$
where $C_{\mathcal{O}} \subset h_0U^{+}U^-$ is a closed subset.  
It follows that
$$
 \dim C_{\cal O} + 1 = \dim \bigl(\overline{\Gamma A_0} \cap \cal O\bigr) \le \dim \overline{\Gamma A_0}. $$
Since
\begin{equation*}
 \overline{\Gamma L} = \overline{\Ga A_0}o =\bigcup_{\mathcal{O}}( C_\mathcal{O} A_0 o) =\bigcup_{\mathcal{O}} C_\mathcal{O} L,
\end{equation*} 
we have $\overline{\Gamma Z_L} = \bigcup_{\mathcal{O}} C_\mathcal{O} Z_L.$
Hence
\begin{equation} \label{eq:dim-union}
 \dim \overline{\Gamma Z_L} \le 2 + \sup \dim C_{\cal O} = 1+ \dim \overline{\Gamma A_0} = 1+ \dim \overline{\Gamma L},
\end{equation}
by  \Cref{ll}. This finishes the proof.
\end{proof}

Therefore, we deduce the following from Theorems \ref{thm:pi} and \ref{pair}:
\begin{theorem}\label{ort}
     Let $ L=hA_0 o \subset Y$ be a geodesic so that
$\ell=\Gamma\ba \Gamma L$ is disjoint from the $2r$-neighborhood of $\beta$ for some $r>0$. Suppose that
$1<\dim \overline{\ell} <2$. Then for all $\b\in B$ with width $\wdt(\b)$ smaller than $r$, we have 
$$\dim \overline{ \Gamma_{\beta, \b} \ba \Gamma_{\beta, \b} Z_L }=\dim \overline{\ell} +1. $$ 

Moreover, there exists a sequence of geodesics $L_i \subset Y$ such that 
$$ \dim \left( \overline{ \Ga_{\beta,\b } \ba \Ga_{\beta, \b } Z_{L_i}} \right) \to 2 \quad\text{as $i\to \infty$} .$$
\end{theorem}

\end{document}